\documentclass[letterpaper,12pt]{amsart}
\advance\oddsidemargin by-0.5in
\advance\evensidemargin by-0.5in
\advance\textwidth by 1in
\usepackage{amsmath,amssymb,mathptmx}
\newcommand{\gateaux}{G\^ateaux}
\newcommand{\frechet}{Fr\'echet}
\newcommand{\real}{\mathbb{R}}

\newcommand{\Z}{\mathbb{Z}}
\newcommand{\e}{\varepsilon}
\newcommand{\F}{R}

\newcommand{\Del}{D}

\newcommand{\mset}{M}
\newcommand{\sset}{W}
\newcommand{\setseq}{\mathfrak S}
\newcommand{\seqn}[2]{#1^{(#2)}}
\newcommand{\pairorder}[6]{(#1,#2)\underset{(#5,#6)}{\le}(#3,#4)}
\renewcommand{\phi}{\varphi}
\theoremstyle{definition}
\newtheorem{theorem}[subsection]{Theorem}
\newtheorem{algorithm}[subsection]{Algorithm}

\theoremstyle{plain}
\newtheorem{lemma}[subsection]{Lemma}
\newtheorem{corollary}[subsection]{Corollary}
\numberwithin{equation}{section}

\newcommand{\bib}{
}
\begin{document}
\title[A compact null set containing a differentiability point]{A compact null set containing a differentiability point of every Lipschitz function}
\thanks{The authors acknowledge support of EPSRC grant EP/D053099/1.}   
\author{Michael Dor\'e}
\address{School of Mathematics,
University of Birmingham,
Edgbaston,
Birmingham B15 2TT,
UK}
\email{M.J.Dore@bham.ac.uk}
\author{Olga Maleva}
\address{School of Mathematics,
University of Birmingham,
Edgbaston,
Birmingham B15 2TT,
UK}
\email{O.Maleva@bham.ac.uk}
\maketitle
\begin{abstract}
We prove that in a Euclidean space of dimension at least two, there
exists a compact set of Lebesgue measure zero such that any
real-valued Lipschitz function defined on the space is differentiable
at some point in the set. Such a set is constructed explicitly.
\end{abstract}

\section{Introduction}
\label{sec1}
\subsection{Background}
A theorem of Lebesgue says that any real-valued
Lipschitz function on the real line is differentiable almost everywhere.
This result is sharp in the sense that for any subset $E$ of the real line with Lebesgue measure zero,
there exists a real-valued Lipschitz function not differentiable
at any point of $E$.
The exact characterisation of the possible sets of non-differentiability of a Lipschitz function $f\colon\real\to\real$ is given in \cite{Z}.

For Lipschitz mappings between Euclidean spaces of higher dimension,
the interplay between Lebesgue null sets and
sets of points of non-differen\-ti\-abili\-ty is less straightforward.
By Rademacher's theorem,  any real-valued Lipschitz mapping
on $\real^n$ is differentiable except on a Lebesgue null set.
However, Preiss \cite{P} gave an example of a
Lebesgue null set $E$ in $\real^n$, for $n \geq 2$,
such that $E$ contains a point of differentiability of
\emph{every} real-valued Lipschitz function on $\real^n$.

In particular, \cite{P} shows that the latter property holds whenever
$E$ is a $G_\delta$-set in $\real^{n}$ - i.e. an intersection of
countably many open sets - such that $E$ contains
all lines passing through two points with rational coordinates.
However, this set is dense in $\real^n$.

In the present paper we construct a
much ``smaller'' set in $\real^{n}$ for $n\geq 2$ --- a
\emph{compact}
Lebesgue null set --- that still captures a point of differentiability of
every Lipschitz function $f\colon  \real^{n} \rightarrow \real$.

It is important to note that though, setting $n = 2$, any
Lipschitz function $f\colon \real^2\to\real$ has points of differentiability
in such an extremely small set as ours, for any Lebesgue null set $E$ in the
plane there is a pair of real-valued Lipschitz functions on $\real^2$
with no common points of differentiability in $E$ \cite{ACP}.

Only a few positive results are known about
the case where the codomain is a space of dimension at least two.
For $n\ge 3$, there exists a Lebesgue null
set in $\real^n$, namely the union of all ``rational hyperplanes'', such that for all $\e>0$
every Lipschitz mapping from $\real^n$ to $\real^{n-1}$ has a point of
$\e$-\frechet{} differentiability in that set; see \cite{PH}.

\subsection{}
Let us say a few words about why the method of \cite{P} does not yield a
set with the properties we are aiming for. Indeed,
\cite[Theorem~6.4]{P} says that every Lipschitz function
defined on $\real^n$ is differentiable at some point of a
$G_\delta$-set $E$ if $E$ satisfies certain conditions, in particular for any two points $u,v\in\real^n$ and any
$\eta>0$, the set $E$ contains a large portion of a path that approximates the line segment $[u,v]$
to within $\eta\|u-v\|$. The closure of such a set $E$ is the
whole space $\real^n$.

There is, however, a stronger version of \cite[Theorem~6.4]{P} that only requires a local version of this condition for the same conclusion to hold: namely
for every $\e>0$ and every $x\in E$ there is a neighbourhood of
$x$ in which any line segment $I$ can be approximated to within $\e |I|$ by a curve in $E$.
Let us explain why the closure of any $G_\delta$-set with this property has
non-empty interior and hence is of positive measure.

Indeed, by this ``local approximation'' property
there is an open ball $B$ intersecting $E$ and
a positive $\eta$,
such that each open  $U\subseteq B$ that intersects $E$ contains
a point $x'\in U\cap E$ with the following property:
any line segment
$I\subseteq B$ through $x'$ of length at most $\eta$ is
pointwise $|I|/2$-close to
a curve inside $E$.
It follows that $E$ is dense in $B$.

Thus in order to construct a closed set of measure zero containing points of
differentiability of every Lipschitz function, we introduce crucial new
steps, outlined in Subsection~\ref{method}.
Before describing our approach we need some preliminaries.

\subsection{Preliminaries}\label{prelim}
Given real Banach spaces $X$ and $Y$,
a mapping $f\colon X\to Y$ is called Lipschitz if there exists
$L \geq 0$ such that $\|f(x)-f(y)\|_Y\le L\|x-y\|_X$ for all
$x,y\in X$.
The smallest such constant $L$ is denoted
$\mathrm{Lip}(f)$.

If $f\colon X\to Y$ is a mapping, then
$f$ is said to be \gateaux{} differentiable at $x_0\in X$ if
there exists a bounded linear operator $D\colon X\to Y$ such that
for every $u\in X$, the limit
\begin{equation}\label{dirderiv}
\lim_{t\to0} \frac{f(x_0+tu)-f(x_0)}{t}
\end{equation}
exists and is equal to $D(u)$. The operator $D$ is called the
\gateaux{} derivative of $f$ at the point $x_0$ and is written $f'(x_0)$.
If this limit exists for some fixed $u$ we say that $f$ has a
directional derivative at $x_0$ in the direction $u$ and denote
the limit by $f'(x_0,u)$.

If $f$ is \gateaux{} differentiable at $x_0$ and the convergence in \eqref{dirderiv} is uniform for $u$ in the unit sphere
$S(X)$ of $X$,
we say that $f$ is \frechet{} differentiable at $x_0$ and
call $f'(x_0)$ the \frechet{} derivative of $f$.

Equivalently, $f$ is \frechet{} differentiable at $x_0$ if we can find a bounded linear operator $f'(x_0)\colon  X \rightarrow Y$ such that for every $\e > 0$
there exists a $\delta > 0$ such that for any $h \in X$ with $\|h\| \leq \delta$ we have
$$
\|f(x_0+h)-f(x_0)-f'(x_0)(h) \| \leq \e \|h\|.
$$

If, on the other hand, we only know this condition for some fixed $\e > 0$ we say that $f$ is $\e$-\frechet{} differentiable at $x_0$.
Note that $f$ is \frechet{} differentiable at $x_0$ if and only if it is
$\e$-\frechet{} differentiable at $x_0$ for every $\e>0$.
In \cite{JLPS,LP} the notion of $\e$-\frechet{} differentiability
is studied in relation to
Lipschitz mappings with the
emphasis on the infinite dimensional case.

In general, \frechet{} differentiability is a strictly stronger property than \gateaux{} differentiability. However
the two notions coincide for Lipschitz functions defined on a
finite dimensional space; see \cite{BL}.

We now make some comments about the porosity property and its connection with
the \frechet{} differentiability of Lipschitz functions.
Recall first that a subset $A$ of a Banach space $X$ is said to be porous at a point $x \in X$
if there exists $\lambda>0$ such that for all $\delta > 0$
there exist $r\le\delta$ and $x'\in B(x,\delta)$ such that
$r>\lambda\|x-x'\|$ and $B(x',r)\cap A=\emptyset$. Here $B(x,\delta)$ denotes an open ball in the Banach space $X$ with centre at $x$ and radius $\delta$.

A set $A \subseteq X$ is called porous if it is porous at every $x\in A$.
A set is said to be $\sigma$-porous if it can be written as a countable union of porous sets.
The family of $\sigma$-porous subsets of $X$ is a $\sigma$-ideal.
A comprehensive survey on porous
and $\sigma$-porous sets can be found in \cite{Zpor}.

Observe that for a non-empty set $A$
the distance function $f(x)=\mathrm{dist}(x,A)$ is Lipschitz with
$\mathrm{Lip}(f) \leq 1$
but is not \frechet{} differentiable
at any porosity point of the set $A$ \cite{BL}. Moreover if $A$ is a $\sigma$-porous subset of a separable
Banach space $X$
we can find a Lipschitz function from $X$ to $\real$
that is not \frechet{} differentiable at any point of $A$. This is proved in \cite{PT} for the case in which $A$ is a countable union of closed porous
sets and, as per remark in \cite[Chapter 6]{BL}, the proof of
\cite[Proposition 14]{PZ} can be used to derive this
statement for an arbitrary $\sigma$-porous set $A$.

The set $S$ we are constructing in this paper
contains a point of differentiability of every Lipschitz function, so
we require $S$ to be non-$\sigma$-porous.
Such a set should also have plenty of non-porosity points.
By the Lebesgue density theorem every
$\sigma$-porous subset of a finite-dimensional space
is of Lebesgue measure zero.
We remark that the $\sigma$-ideal of $\sigma$-porous sets is
a proper subset of that of Lebesgue null sets.
In order to arrive at an appropriate set that is not $\sigma$-porous, has no porosity points and whose
closure has measure zero, we use ideas similar to those in \cite{Z76,Z98,ZP}.

\subsection{Construction}\label{method}
We now outline the method we use to prove that the set $S$
we construct contains a differentiability point of every Lipschitz
function.

Given a Lipschitz function $f\colon \real^n\to\real$, we first find a
point $x \in S$ and a direction $e \in S^{n-1}$, the unit sphere of
$\real^{n}$, such that the directional derivative $f'(x,e)$ exists and
is locally maximal in the sense that if $\e > 0$, $x'$ is a nearby
point of $S$, $e' \in S^{n-1}$ is a direction and $(x',e')$ satisfies
appropriate constraints, then $f'(x',e') < f'(x,e)+\e$.

We then prove $f$ is differentiable at $x$ with derivative
$$
D(u) = f'(x,e)\langle u,e \rangle.
$$
A heuristic outline goes as follows. Assume this is not true. Find
$\eta > 0$ and a vector $\lambda$ with small norm
such that $|f(x+\lambda)-f(x)-f'(x,e)\langle \lambda,e \rangle| >
\eta\|\lambda\|$.
Then construct an auxiliary point $x+h$ lying near the
line $x+\real e$ and calculate the
ratio
$$
\frac{|f(x+\lambda)-f(x+h)|}{\|\lambda-h\|}.
$$
We find that this is at least $f'(x,e)+\e$ for some $\e > 0$. By using
an appropriate mean value theorem
\cite[Lemma 3.4]{P},
it is possible to find a point $x'$ on the line segment
$[x+h,x+\lambda]$ and a direction $e' \in S^{n-1}$ such that
$f'(x',e') \geq f'(x,e)+\e$ and $(x',e')$ satisfies the required
constraints. This contradicts the local maximality of $f'(x,e)$ and so
$f$ is differentiable at $x$.

Since $f'(x,e)$ is only required to be locally maximal for $x$ in the
set $S$, it is necessary to ensure the above line segment
$[x+h,x+\lambda]$ lies in $S$, if we are to get a contradiction. It is
therefore vital to construct $S$ so that it contains lots of line segments.

Crucially, instead of just one set, we introduce a hierarchy of
closed null sets
$\mset_i$, indexed by sequences $i$ of real numbers that are subject
to a certain partial ordering. For any point $x$ in $\mset_i$
the required line segments $[x+h,x+\lambda]$
can be found in every set $\mset_{j}$ where
$j$ is greater than $i$ in the sense of the partial order.
Subsequently we prove in Corollary~\ref{corfinal} that each set $\mset_i$
contains a point of differentiability of every Lipschitz function.
The
desired set $S$ can then be taken equal to the intersection of
any of the $\mset_i$ with a closed ball.

\subsection{Structure of the paper}
Section~\ref{sec2} is devoted to the description of the partial
ordered set and the layers $\mset_i$. The existence of line segments close to any point in a previous layer is verified in
Theorem~\ref{resultofsection}. In Section~\ref{sec5} we will show that
this condition is sufficient for any Lipschitz function to have a point of
differentiability in each layer.

In Section~\ref{sec3} we show in detail how to arrive at a pair
$(x,e)$ with ``almost maximal'' directional derivative $f'(x,e)$.
By a modification of the method in \cite{P} we construct a sequence of
points $x_m$ and directions $e_m \in S^{n-1}$ such that $f$
has a directional derivative $f'(x_{m},e_{m})$
that is almost maximal, subject to some constraints.
We then argue that $(x_m)$ and $(e_m)$ both converge and that the
directional derivative $f'(x,e)$ at $x = \lim_{m \to \infty} x_m$ in
the direction $e = \lim_{m \to \infty} e_m$ is locally maximal in the
required sense. We eventually show $x$ is a point of differentiability of
$f$.

The convergence of $(x_m)$ is achieved simply by choosing $x_{m+1}$
close to $x_m$. The convergence of $e_m$ is more subtle; we obtain
this by altering the function by an appropriate small linear piece at
each stage of the iteration. Then picking $(x_m,e_m)$ such that the
$m$th function $f_m$ has almost maximal directional derivative
$f_m'(x_m,e_m)$ can be shown to guarantee that the sequence $(e_m)$ is
Cauchy.

In Section~\ref{sec4} we introduce a Differentiability Lemma, ~\ref{lemdiff}, showing that under certain conditions such a pair $(x,e)$, with $f'(x,e)$ almost maximal, gives a point $x$ of \frechet{} differentiability of $f$.

Finally in Section~\ref{sec5} we verify the conditions of this Differentiability
Lemma~\ref{lemdiff} for the pair $(x,e)$ constructed in Section~\ref{sec3}, using the results of Section~\ref{sec2}. This completes the proof.

\subsection{Related questions}
To conclude the introduction let us observe the following.
Independently of our
construction, one can deduce
from  \cite{B,HMZW}
that there exists a non-empty Lebesgue null set $E$ in the plane with a
weaker property:
$E$ is $F_\sigma$ - i.e. a countable union of closed sets - and
contains a point of sub-differentiability of every real-valued Lipschitz
function.

Indeed, in \cite{B} it is proved that
there exist a non-empty open set $G\subseteq\real^2$, a differentiable
function $f\colon G\to\real$ and a non-empty open set
$\Omega\subseteq\real^2$ for which there exists a point $p\in G$ such
that the gradient $\nabla f(p)\in\Omega$ but $\nabla f(q)\notin\Omega$
for almost
all $q\in G$, in the sense of two dimensional Lebesgue measure.
In other words, the set $E=(\nabla f)^{-1}(\Omega)\cap G$ is a
non-empty set of Lebesgue measure zero.
Note that $\nabla f$ is a Baire-$1$ function; therefore
the set $E$, which is a preimage of an open set,
is an $F_\sigma$ set. Now \cite[Lemma 4]{HMZW}
implies that any Lipschitz function $h\colon \real^2\to\real$ has a point of
sub-differentiability in $E$.

\subsection{Acknowledgement}
The authors wish to thank Professor David Preiss for stimulating
discussions.

\section{The set}

\label{sec2}
\label{secconstr}
Let $(N_r)_{r\geq 1}$ be a sequence
of odd integers such that
$N_r > 1$, $N_r\to\infty$ and $\sum\frac{1}{N_r^2}=\infty$. Let $\setseq$
be the set of all sequences
$i=(\seqn{i}{r})_{r \geq 1}$ of real numbers with
$1 \leq \seqn{i}{r} < N_{r}$ for all $r$
and $\seqn{i}{r}/N_{r} \rightarrow 0$ as
$r \rightarrow \infty$.

We define a relation $\preceq$ on $\setseq$ by
$$
i \prec j \textrm{ if }(\forall r)(\seqn{i}{r} > \seqn{j}{r}) \textrm{ and }
\seqn{i}{r} / \seqn{j}{r}
\rightarrow \infty \textrm{ as }r \rightarrow \infty
$$
and
$$
i\preceq j
\textrm{ if } i\prec j
\textrm{ or } i=j.
$$
For $i,j\in \setseq$ such that $i\prec j$, we denote by $(i,j)$ the set
$\{k\in \setseq\colon i\prec k\prec j\}$ and by
$[i,j]$ the set
$\{k\in \setseq\colon i\preceq k\preceq j\}$.

Recall that a partially ordered set - or poset - is a pair $(X,\leq)$ where $X$ is a set and $\leq$ is a relation on $X$ such that $x \leq x$ for all $x \in X$, if $x \leq y$ and $y \leq x$ for $x,y \in X$ then necessarily $x = y$ and finally if $x,y,z \in X$ with $x \leq y$ and $y \leq z$ then $x \leq z$.

A chain in a poset $(X,\leq)$ is a subset $C \subseteq X$ such that
for any $x,y \in C$ we have $x \leq y$ or $y \leq x$. We say $(X,
\leq)$ is chain complete if every non-empty chain $C \subseteq X$ has
a least upper bound - or ``supremum'' - in $X$.

We write $x < y$ if $x \leq y$ and $x \neq y$. We call $(X,\leq)$ dense if whenever $x,y \in X$ with $x < y$ we can find $z \in X$ such that $x < z < y$. Finally, recall that an element $x$ of $X$ is minimal if there does not exist $y$ with $y < x$.

The following lemma summarises basic properties of $(\setseq,\preceq)$.

\begin{lemma}
\label{posetproperties}
$(\setseq,\preceq)$ is a non-empty partially ordered set that is chain complete, dense and has no minimal element.
\end{lemma}
\begin{proof}
It is readily verified that $(\setseq,\preceq)$ is a poset and that $\setseq \neq \emptyset$ since it contains the element $(1,1,1,\dots)$. Given a non-empty chain $C = \{i_{\alpha}\mid\alpha \in A\}$
in $\setseq$, the supremum of $C$ exists and is given by $i\in \setseq$ where
$\seqn{i}{r}  = \inf_{\alpha \in A}\seqn{i_\alpha}{r}$; hence $(\setseq,\preceq)$ is chain complete. To see that $(\setseq,\preceq)$ is dense, note that if $i,j \in \setseq$ with $i \prec j$ then $i \prec k \prec j$ where $k \in \setseq$ is given by $\seqn{k}{r} = \sqrt{\seqn{i}{r}\seqn{j}{r}}$. Finally given $l \in \setseq$, we can find $m \in \setseq$ with $m \prec l$ by taking $\seqn{m}{r} = \sqrt{\seqn{l}{r}N_{r}}$. Therefore $(\setseq,\preceq)$ has no minimal element.
This completes the proof of the lemma.
\end{proof}

We begin by working in the plane $\real^2$.

Denote the inner product $\langle , \rangle$ and the Euclidean norm $\|\cdot\|$. Write $B(x,\delta)$ for an open ball in $(\real^2,\|\cdot\|)$ with centre $x \in \real^2$ and radius $\delta > 0$. Further let $B_{\infty}(c,d/2)$ be an open ball in $(\real^2,\| \cdot \| _{\infty})$, i.e. an open square with centre $c \in \real^2$ and side $d > 0$. Finally, given $x,y \in \real^2$ we use $[x,y]$ to denote the closed line segment
\begin{equation*}
\{(1-\lambda)x + \lambda y \mid 0 \leq \lambda \leq 1\} \subseteq \real^2.
\end{equation*}

Let $d_0=1$. For each $r \geq 1$ set $d_r = \frac{1}{N_1N_2\dots N_{r}}$ and define the lattice $C_r \subseteq \real^2$:
\begin{equation}\label{deflat}
C_r=d_{r-1}\left(\left(\frac{1}{2},\frac{1}{2}\right)+\Z^2\right).
\end{equation}

Suppose now $i \in \setseq$. Define the set $\sset_i \subseteq \real^{2}$ by
\begin{equation}\label{defsseti}
\sset_i
=\real^2\setminus
\bigcup_{r=1}^\infty
\bigcup_{c\in C_r}
B_{\infty}\left(c,\frac{1}{2}\seqn{i}{r}d_r\right).
\end{equation}

Note that each $\sset_i$ is a closed subset of the plane and
$\sset_i\subseteq\sset_j$ if $i\preceq j$.
From $\seqn{i}{r} < N_r$ we see that $\sset_i \neq \emptyset$ - for example $(0,0) \in \sset_i$. We now claim that the Lebesgue measure of $\sset_i$ is equal to $0$.

For each $r \geq 0$ we define sets $\Del_r$ and $\F_r$ of disjoint open squares of side $d_r$ as follows.
Recall $d_0 = 1$. Let $\Del_0$ be the empty-set and $\F_0 = \{U\}$ be a singleton
comprising the open unit square:
$$
U=\{(x,y) \in \real^2 \mid 0 < x,y < 1\}.
$$

Divide each square in the set $\F_{r-1}$ into an $N_r\times N_r$ grid.
Let $\Del_r$ comprise the central open squares of the grids
and let $\F_r$ comprise all the remaining open squares.
By induction each square in $\Del_r$ and $\F_r$ has side $d_r$ and
the centres of the squares in $\Del_r$ belong to the
lattice $C_r$.
For each $m \geq 1$ we have from
\eqref{defsseti} and $\seqn{i}{r} \geq 1$,
\begin{equation*}
\sset_i \subseteq
\real^2\setminus
\bigcup_{r=1}^m \bigcup_{c\in C_r} B_{\infty}\left(c,\frac{1}{2}d_r\right)
\end{equation*}
so that
\begin{equation*}
\sset_i \cap {U}
\subseteq \overline{U} \setminus \bigcup_{r=1}^m \bigcup \Del_r
= \overline{\bigcup \F_m},
\end{equation*}
and, as the cardinality of the set $\F_m$ is equal to $(N_1^2-1)\dots(N_m^2-1)$ and each square in $\F_m$ has area $d_m^{2}$,
we can estimate the Lebesgue measure of $\sset_i \cap {U}$:
\begin{equation*}
|\sset_i \cap {U}|
\leq \left(1-\frac{1}{N_1^2}\right) \dots \left(1-\frac{1}{N_m^2}\right).
\end{equation*}
This tends to $0$ as $m \rightarrow \infty$, because
$\sum \frac{1}{N_r^2} = \infty$. Therefore the Lebesgue measure
$|\sset_i \cap {U}| = 0$.
Furthermore, from \eqref{deflat} and \eqref{defsseti}, $\sset_i$ is invariant under translations by the lattice $\mathbb{Z}^{2}$. Hence $|\sset_i| = 0$ for every $i \in \setseq$.

Let
$$
\sset
=
\bigcup_{\overset{i\in\setseq}{i\prec(1,1,1,\dots)}} \sset_i.
$$
As $(1,1,1,\dots)$ is not minimal and $\sset_i \neq \emptyset$ for any $i \in \setseq$, we observe $\sset$ is not empty.
The following theorem
now proves that for any point
$x\in \sset$ there are line segments inside $\sset$
with directions that cover a dense subset of the unit circle. We say $e
= (e_{1},e_{2}) \in S^{1}$ has rational slope if there exists
$(p,q)\in \Z^2 \setminus \{(0,0)\}$ with $pe_{1} = qe_{2}$.

\begin{theorem}
\label{lemmain}
For any $i,j \in\setseq$ with $i \prec j$, $\e > 0$ and
$e\in S^1$ with rational slope there exists
$\delta_0=\delta_0(i,j,\e,e)>0$ such that whenever $x \in \sset_{i}$ and
$\delta \in (0,\delta_0)$,
there is a line segment $[x',x'+\delta e] \subseteq \sset_{j}$ where
$\|x'-x\| \leq \e \delta$.
\end{theorem}
\begin{proof}
First we note that without loss of generality we may assume that
$\e \leq 1$ and
$|e_2|\le |e_1|$ where $e = (e_{1},e_{2})$. Write $e_2/e_1=p/q$ with $p,q \in \mathbb{Z}$ and $q > 0$. Now observe that if $y \in \real^2$ then the line $y + \real e$ has gradient $p/q \in [-1,1]$ and if it
intersects the square
$B_{\infty}(c,d/2)$,
\begin{equation}
\label{intersectioncriterion}
\left|(y_2-c_2)-\frac{p}{q}\left(y_1-c_1\right)\right|< d
\end{equation}
where $y = (y_1,y_2)$ and $c = (c_1,c_2)$.

From $i\prec j$, we have $\sup_m\frac{\seqn{j}{m}}{\seqn{i}{m}}<1$ so that we can find
$\psi>0$ such that
$\frac{\seqn{j}{m}}{\seqn{i}{m}}\le1-\psi$ for all $m$.
Put $\rho_m=\seqn{i}{m} d_m \psi/4$.
Since $d_m=N_{m+1}d_{m+1}$ and $\seqn{i}{m} \geq 1$ for each $m\geq 1$,
$$
\rho_m/\rho_{m+1}
=(\seqn{i}{m}N_{m+1})/\seqn{i}{m+1}
\geq
\inf_m
\frac{N_{m+1}}{\seqn{i}{m+1}}
> 1
$$
so that $\rho_m \searrow 0$.
Let $k_0$ be such that
\begin{equation}
\label{defm0}
\begin{cases}
{\seqn{j}{m}}/{\seqn{i}{m}} \leq \e\psi/16\\
{\seqn{j}{m}}/{N_m}\leq(5q)^{-1}
\end{cases}
\textrm{ for all }
m\ge k_0.
\end{equation}
We set $\delta_0=\rho_{k_0}$ and let $\delta\in(0, \delta_0)$.
Since $\rho_k\to0$, there exists $k\ge k_0$ such that
$\rho_k\ge\delta>\rho_{k+1}$.

Let $C_m$ be given by \eqref{deflat} and set
\begin{equation*}
T_m = \bigcup_{c \in C_m} B_\infty(c,\seqn{j}{m} d_m/2)
\end{equation*}
so that $\sset_j = \bigcap_{m \geq 1} (\real^2\setminus T_m)$.

Fix any point $x \in \sset_i$. Define the line $\ell_\lambda = x+(0,\lambda) + \real e \subseteq \real^2$ to be the vertical shift of $x + \real e$ by $\lambda$. We claim that if $m \geq k+1$ and $I \subseteq \real$ is a closed interval of length at least $4\seqn{j}{m}d_m$ we can find a closed subinterval $I' \subseteq I$ of length $\seqn{j}{m}d_m$ such that the line $\ell_{\lambda}$ does not intersect $T_m$ for any $\lambda \in I'$.

Take $I = [a,b]$. We may assume there exists $\lambda \in [a,a+\seqn{j}{m}d_m]$ such that $\ell_\lambda$ intersects $B_\infty(c,\seqn{j}{m}d_m/2)$ for some $c \in C_m$; if not we can take $I' = [a,a+\seqn{j}{m}d_m]$. Write $c = (c_1,c_2)$ and $x = (x_1,x_2)$. Note that from \eqref{intersectioncriterion} we have
\begin{equation*}
\left|(x_2+\lambda-c_2)-\frac{p}{q}\left(x_1-c_1\right)\right|< \seqn{j}{m}d_m.
\end{equation*}
Let $I' = [\lambda+2\seqn{j}{m}d_m,\lambda+3\seqn{j}{m}d_m] \subseteq I$.
Suppose that $\lambda' \in I'$ and that $c' \in C_m$. We may write $c' = (c_1',c_2') = (c_1,c_2)+(l_1,l_2)d_{m-1}$ where $l_1,l_2 \in \mathbb{Z}$. Then if $pl_1 \neq ql_2$,
\begin{align*}
\Bigl|(x_2&+\lambda'-c_2')-\frac{p}{q}(x_1-c_1')\Bigr|\\
&\geq
d_{m-1}\left|\frac{pl_1-ql_2}{q}\right|
-\left|(x_2+\lambda-c_2)-\frac{p}{q}(x_1-c_1)\right| - |\lambda'-\lambda| > \seqn{j}{m}d_m
\end{align*}
as $|pl_1-ql_2| \geq 1$ and $d_{m-1} = N_m d_m \geq 5q \seqn{j}{m} d_m$ from \eqref{defm0}.
On the other hand if $pl_1 = ql_2$ the same inequality holds as
\begin{align*}
\Bigl|(x_2+\lambda'-c_2')-&\frac{p}{q}\left(x_1-c_1'\right)\Bigr|\\
&\geq
|\lambda'-\lambda| -
\Bigl| (x_2+\lambda-c_2)-
\frac{p}{q}\left(x_1-c_1\right) \Bigr|
>
\seqn{j}{m}d_m.
\end{align*}
Therefore by \eqref{intersectioncriterion} the line $\ell_{\lambda'}$ does not intersect $B_\infty(c',\seqn{j}{m}d_m/2)$ for any $c' \in C_m$ and any $\lambda' \in I'$. Hence the claim.

Note that for $m \geq k+1$ we have $\seqn{j}{m}d_{m} \geq 4\seqn{j}{m+1}d_{m+1}$ from \eqref{defm0}. Subsequently, by the previous claim, we may construct a nested sequence of closed intervals
\begin{equation*}
[0,4\seqn{j}{k+1}d_{k+1}] \supseteq I_{k+1} \supseteq I_{k+2} \supseteq \dots
\end{equation*}
such that $|I_m| = \seqn{j}{m}d_m$ and $\ell_\lambda$ does not intersect $T_m$ for $\lambda \in I_m$.

Picking $\lambda \in \bigcap_{m \geq k+1} I_m$ we have
\begin{equation*}
0 \leq \lambda \leq 4\seqn{j}{k+1}d_{k+1} \leq \frac{\seqn{i}{k+1}\psi \e}{4}d_{k+1} = \e \rho_{k+1} < \e \delta
\end{equation*}
using \eqref{defm0} again.

Set $x' = x+(0,\lambda)$ so that $\|x'-x\| = \lambda < \e \delta$. Note that $[x',x'+\delta e]$ does not intersect $T_m$ for $m \geq k+1$ as $[x',x'+\delta e] \subseteq \ell_{\lambda}$ and $\lambda \in I_m$. Now suppose $m \leq k$. From $\e \leq 1$ we have $\lambda \leq \delta \leq \rho_k$.
If $c \in C_m$ then we observe that $[x',x'+\delta e]$ does not
intersect $B_\infty(c,\seqn{j}{m}d_m/2)$ as $x\in \sset_{i}$ is
outside
$B_\infty(c,\seqn{i}{m}d_m/2)$ and
\begin{align*}
\lambda+\delta
\leq 2\rho_k
\le
2\rho_m
&=
\frac{1}{2} \seqn{i}{m}d_m\psi
\leq
\frac{1}{2}
\seqn{i}{m}d_m\left(1-\frac{\seqn{j}{m}}{\seqn{i}{m}}\right)\\
&=
\frac{1}{2}(\seqn{i}{m}d_m-\seqn{j}{m}d_m).
\end{align*}
Therefore $[x',x'+\delta e]$ does not intersect $T_m$ for any $m \geq 1$ so that $[x',x'+\delta e] \subseteq \sset_j$. This finishes the proof.
\end{proof}

We now give a simple geometric lemma and then
prove some corollaries to Theorem~\ref{lemmain}.
Given $e = (e^{1},e^{2}) \in S^{1}$ we define $e^{\perp} = (-e^2,e^1)$ so that $\langle e^{\perp},e \rangle = 0$ for any $e \in S^1$ and, given $x_0 \in \real^2$ and $e_0 \in S^1$, then $x \in \real^2$ lies on the line $x_0 + \real e_0$ if and only if $\langle e_0^{\perp},x \rangle = \langle e_0^{\perp},x_0 \rangle$.
\begin{lemma}\label{lemcross}
Suppose that $x_{1},x_{2} \in \real^{2}$, $e_{1},e_{2} \in S^{1}$,
$\alpha_{1},\alpha_{2} > 0$, the line segments $l_{1}$,
$l_{2}$ given by $l_{m} = [x_{m},x_{m}+\alpha_{m}e_{m}]$ intersect at
$x_{3} \in \real^{2}$ and that
\begin{equation}\label{crosspoint}
[x_{3}-\alpha e_{m},x_{3}+\alpha e_{m}] \subseteq l_{m},
\qquad
(m = 1,2)
\end{equation}
where $\alpha > 0$.
If $x_{1}'$, $x_{2}' \in \real^{2}$ and $e_{1}'$,
$e_{2}' \in S^{1}$
are such that
\begin{align}
\label{firstineq}\|x_{m}'-x_{m}\| &
\leq \frac{\alpha}{16} |\langle e_2^{\perp},e_1 \rangle|
\qquad\textrm{and}\\
\label{secondineq}\|e_{m}'-e_{m}\| &
\leq \frac{\alpha}{8(\alpha_{1}+\alpha_{2})}|\langle e_2^{\perp},e_1
\rangle|
\end{align}
for $m = 1,2$, then the line segments $l_{1}'$, $l_{2}'$ given by
$l_{m}' = [x_{m}',x_{m}'+\alpha_{m}e_{m}']$ intersect at a point
$x_{3}'\in \real^{2}$ with $\|x_{3}'-x_{3}\| \leq \alpha$.
\end{lemma}

\begin{proof}
As $\langle e_2^{\perp},e_1 \rangle = -\langle e_1^{\perp},e_2
\rangle$ we may assume, without loss of generality, that
the inner product
$\langle e_2^{\perp},e_1 \rangle$
is non-negative.
From \eqref{crosspoint}
we can write $x_{3}=x_{m}+\lambda_{m}e_{m}$ for $m = 1,2$ with
$\alpha\leq \lambda_{m} \leq \alpha_{m}-\alpha$.
Now note that as $x_1+\lambda_1 e_1 \in l_2$ we have
$$
\langle e_2^{\perp},x_1+\lambda_1 e_1 \rangle = \langle
e_2^{\perp},x_2 \rangle
$$
so that
\begin{equation}\label{equality}
\langle e_2^{\perp},x_{1}+(\lambda_{1}+\pi \frac{1}{2}\alpha)e_{1} \rangle - \langle e_2^{\perp},x_2 \rangle = \pi\frac{\alpha}{2} \langle e_2^{\perp},e_1 \rangle
\end{equation}
for $\pi = \pm 1$. Using \eqref{firstineq} and \eqref{secondineq} we quickly obtain from \eqref{equality}
\begin{align}\label{eq.pos}
&\langle e_2'^{\perp},x_{1}'+(\lambda_{1}+\frac{1}{2}\alpha)e_{1}'
  \rangle - \langle e_2'^{\perp},x_2' \rangle \geq 0\\
\label{eq.neg}
\text{and }&\langle e_2'^{\perp},x_{1}'+(\lambda_{1}-\frac{1}{2}\alpha)e_{1}' \rangle - \langle e_2'^{\perp},x_2' \rangle \leq 0.
\end{align}
Indeed, for $\pi=\pm1$,
\begin{align*}
\Bigl(\langle &e_2'^{\perp},x_{1}'+(\lambda_{1}+\pi\frac{1}{2}\alpha)e_{1}' \rangle - \langle e_2'^{\perp},x_2' \rangle\Bigr)
-
\Bigl(
\langle e_2^{\perp},x_{1}+(\lambda_{1}+\pi \frac{1}{2}\alpha)e_{1}
\rangle - \langle e_2^{\perp},x_2 \rangle \Bigr)\\
=
\langle &e_2'^{\perp},(x_1'-x_1)-(x_2'-x_2)+
(\lambda_{1}+\pi \frac{1}{2}\alpha)(e_1'-e_1)\rangle\\
&+
\langle (e_2'^\perp-e_2^\perp),(x_1-x_2+(\lambda_{1}+\pi \frac{1}{2}\alpha)e_1\rangle;
\end{align*}
the norm of the first term  is bounded by
\begin{align*}
\|x_{1}'-x_{1}\| + &\|x_{2}'-x_{2}\| + |\lambda_{1}+\pi\frac{1}{2}\alpha|\cdot
  \|e_{1}'-e_{1}\| \\
&\leq 2\frac{\alpha}{16}
\langle e_{2}^\perp, e_{1}\rangle +
  \alpha_{1}\frac{\alpha}{8(\alpha_{1}+\alpha_{2})}
\langle e_{2}^\perp, e_{1}\rangle
 \leq \frac{\alpha}{4}\langle e_{2}^\perp, e_{1}\rangle,
\end{align*}
and the norm of the second term is bounded by
\begin{multline*}
\|e_{2}'-e_{2}\| (\|x_{1}-x_{2}\| + |\lambda_{1}+\pi\frac{1}{2}\alpha|)
\leq \frac{\alpha}{8(\alpha_{1}+\alpha_{2})}\langle e_{2}^\perp, e_{1}\rangle
((\alpha_{1}+\alpha_{2})+\alpha_{1})\\
\leq \frac{\alpha}{4}\langle e_{2}^\perp, e_{1}\rangle.
\end{multline*}

Hence by \eqref{eq.pos} and \eqref{eq.neg} there exists
\begin{equation}\label{defx'}
x_{3}' \in
[x_{1}'+(\lambda_{1}-\frac{1}{2}\alpha)e_{1}',
x_{1}' +(\lambda_{1}+\frac{1}{2}\alpha) e_{1}'] \subseteq l_{1}'
\end{equation}
with $\langle e_{2}'^\perp, x_{3}'\rangle = \langle e_{2}'^\perp, x_{2}'\rangle$
so that we can write
\begin{equation}\label{deflambda2'}
x_{3}' = x_{2}'+\lambda_{2}'e_{2}'
\end{equation}
for some
$\lambda_{2}' \in \real$. Since $x_{3} = x_{1}+\lambda_{1}e_{1}$
and \eqref{defx'} imply
\begin{equation*}
\|x_{3}'-x_{3}\| \leq \|x_{1}'-x_{1}\| + \lambda_{1} \|e_{1}'-e_{1}\|
+ \frac{1}{2}\alpha \|e_{1}'\| \leq \frac{3}{4}\alpha
\end{equation*}
and $x_{3} = x_{2}+\lambda_{2}e_{2}$  and \eqref{deflambda2'} imply
\begin{equation*}
\|x_{3}'-x_{3}\|
\geq |\lambda_{2}'-\lambda_{2}| - \|x_{2}'-x_{2}\| - \lambda_{2}\|e_{2}'-e_{2}\|\geq |\lambda_{2}'-\lambda_{2}| - \frac{1}{4}\alpha,
\end{equation*}
we get
\begin{equation*}
|\lambda_{2}'-\lambda_{2}|
\leq
\frac{3}{4}\alpha + \frac{1}{4}\alpha =
\alpha.
\end{equation*}
It follows that
\begin{equation*}
x_{3}' \in [x_{2}'+(\lambda_{2}-\alpha) e_{2}',x_{2}'+
(\lambda_{2}+\alpha)e_{2}'] \subseteq l_{2}'
\end{equation*}
since $\alpha \leq \lambda_{2} \leq \alpha_{2}-\alpha$. Therefore
$x_{3}'\in l_{1}' \cap l_{2}'$ with
$\|x_{3}'-x_{3}\| \leq \frac{3}{4}\alpha < \alpha$ as required.
\end{proof}

\begin{corollary}\label{corint}
Suppose $i,j \in\setseq$ with $i \prec j$ and $\e > 0$.
\begin{enumerate}
\item
There exists
$\delta_1=\delta_1(i,j,\e) > 0$ such that whenever
$\delta\in(0, \delta_{1})$,
$x \in \sset_{i}$ and $e \in S^{1}$,
there exists a line segment $[x',x'+\delta e'] \subseteq \sset_{j}$
where $x' \in \real^{2}$, $e' \in S^{1}$ with
$\|x'-x\| \leq \e\delta$ and $\|e'-e\| \leq \e$.
\item\label{corint-2}
There exists $\delta_{2} = \delta_{2}(i,j,\e) > 0$ such that
whenever $\delta \in(0, \delta_{2})$, $x \in \sset_{i}$, $u \in
B(x,\delta)$ and $e \in S^{1}$ there exists a line segment $[u',u'+\delta
  e'] \subseteq \sset_{j}$ where $u' \in \real^{2}$, $e' \in S^{1}$
with $\|u'-u\| \leq \e \delta$ and $\|e'-e\| \leq \e$.
\item\label{corint-3}
For $v_{1},v_{2},v_{3} \in \real^{2}$ there exists
$\delta_3 = \delta_3(i,j,\e,v_{1},v_{2},v_{3}) > 0$
such that
whenever $\delta\in(0, \delta_{3})$ and
$x\in\sset_i$
there exist  $v_{1}',v_{2}',v_{3}'\in\real^2$
such that $\|v_{m}'-v_{m}\| \leq \e$ and
\begin{equation*}
[x+\delta  v_{1}',x+\delta v_{3}'] \cup [x+\delta v_{3}',x+\delta v_{2}'] \subseteq \sset_{j}.
\end{equation*}
\item\label{corint-4}
There exists $\delta_4 = \delta_4(i,j,\e) > 0$ such that whenever $\delta \in (0,\delta_4)$, $v_1,v_2,v_3$ are in the closed unit ball $D^2$ of $\real^2$ and $x \in \sset_i$ there exist $v_1',v_2',v_3' \in \real^2$ such that $\|v_m'-v_m\| \leq \e$ and
\begin{equation*}
[x+\delta  v_{1}',x+\delta v_{3}'] \cup [x+\delta v_{3}',x+\delta v_{2}'] \subseteq \sset_{j}.
\end{equation*}
\end{enumerate}
\end{corollary}
\begin{proof}
1. We can find a finite collection of unit vectors in the plane
$$
e_{1},e_{2},\dots,e_{r} \in S^{1}
$$
with rational slopes such that
$S^{1}\subseteq \bigcup_{1\le s\le r} B(e_{s},\e)$. Let
\begin{equation*}
\delta_1=\min_{1\le s\le r}\delta_0(i,j,\e,e_{s}),
\end{equation*}
where $\delta_0$ is given by Theorem~\ref{lemmain}. Then for any $\delta \in(0, \delta_{1})$, $x \in \sset_{i}$ and $e \in S^{1}$
find $e_s$ with $\|e_{s}-e\| \leq \e$. As
$\delta < \delta_{0}(i,j,\e,e_s)$ there exists a line segment
$[x',x'+\delta e_s] \subseteq \sset_{j}$ with
$\|x'-x\| \leq \e \delta$ as required.

2. Pick any $k\in\setseq$ with $i\prec k\prec j$.
Let
$$
\delta_2=\min(\delta_1(i,k,\e/3),\delta_1(k,j,\e/3)).
$$
Suppose that $\delta \in(0, \delta_{2})$ and $u \in B(x,\delta)$. We can
write $u = x+\delta' f$ with $0 \leq \delta' < \delta$ and
$f \in S^{1}$. Then there exists $x' \in \real^{2}$, $f' \in S^{1}$ such that
$[x',x'+\delta f'] \subseteq \sset_k$ with
$\|x'-x\| \leq \e \delta/3$ and
$\|f'-f\| \leq \e/3$. As $x'+\delta'f' \in\sset_k$
we can find $u' \in \real^{2}$, $e' \in S^{1}$ such that
$[u',u'+\delta e'] \subseteq \sset_j$ with
$\|u'-(x'+\delta' f')\|\leq \e \delta/3$ and
$\|e'-e\| \leq \e/3$. Then
\begin{equation*}
\|u'-u\|
\leq \|u'-(x'+\delta' f')\| + \|x'-x\| + \delta' \|f'-f\|
\leq
\e \delta
\end{equation*}
as required.

3. Without loss of generality, we may assume that $v_{1},v_{2},v_{3}$ are
not collinear and that $\|v_{1}\|,\|v_{2}\|,\|v_{3}\| \leq
\frac{1}{4}$. Write
\begin{equation}\label{eqvr}
v_{3} = v_{1}+t_{1}e_{1} = v_{2}+t_{2}e_{2}
\end{equation}
where $0 < t_{1},t_{2} \leq \frac{1}{2}$ and $e_{1},e_{2} \in
S^{1}$. As $v_{1}$, $v_{2}$, $v_{3}$ are not collinear, the vectors
$e_{1}$ and $e_{2}$ are not parallel so that $\langle e_{2}^\perp,e_1 \rangle
\neq 0$.
We may assume $\e \leq t_{1},t_{2}$. Set
\begin{equation*}
\delta_{3}
=
\delta_{2}(i,j,\eta),
\end{equation*}
where
$\eta = \frac{1}{16} |\langle e_{2}^\perp,e_1\rangle| \e$.
Let $\delta \in(0, \delta_{3})$.
Write
\begin{equation}\label{defxr}
x_{m} = x+\delta  v_{m}
\qquad
(m=1,2)
\end{equation}
and put
$l_{m} = [x_{m},x_{m} + 2\delta t_{m} e_{m}]$.
As $\|x_{m}-x\| < \delta_{3}$, by part \eqref{corint-2} of this Corollary
we can find
$x_{1}',x_{2}' \in \real^{2}$ and $e_{1}',e_{2}' \in S^{1}$ with
$\|x_{m}'-x_{m}\| \leq \eta \delta$, $\|e_{m}'-e_{m}\| \leq \eta$ and
$[x_{m}',x_{m}'+\delta e_{m}'] \subseteq \sset_{j}$ for
$m =1,2$. Then as $t_{1},t_{2} \leq \frac{1}{2}$ we have
$l_{m}' \subseteq \sset_{j}$ where
$l_{m}' = [x_{m}',x_{m}' + 2\delta t_{m} e_{m}']$
for $m = 1,2$.

Note that
\eqref{eqvr} and \eqref{defxr} imply that
$x+\delta v_{3} = x_{m}+\delta t_{m} e_{m}$ for $m=1,2$.
Therefore $x_3=x+\delta v_{3}$ is a point of intersection of
$l_{1}$ and $l_{2}$.
The conditions of
Lemma~\ref{lemcross}
are readily
verified with $\alpha_{m} = 2\delta t_m$ and $\alpha = \e
\delta$ so that $l_{1}',l_{2}'$ intersect at a point $x_{3}'$ with
$\|x_{3}'-x_{3}\| \leq \e \delta$. Writing now $x_{m}' = x +
\delta v_{m}'$ for $m = 1,2,3$ we have $\|v_{m}'-v_{m}\| \leq
\e$, since $\|x_{m}'-x_{m}\| \leq \e \delta$, and
\begin{equation*}
[x+\delta v_{1}',x+\delta v_{3}'] \cup [x+\delta  v_{3}',x+\delta v_{2}']
\subseteq \sset_{j}.
\end{equation*}

4. Take $w_1,w_2,\dots,w_r$ in $D^2$ with $D^2 \subseteq \bigcup_{1 \leq s \leq r}B(w_s,\e/2)$. Set
\begin{equation*}
\delta_4=\min_{1\le s_1,s_2,s_3\le r}
\delta_3(i,j,\e/2,w_{s_1},w_{s_2},w_{s_3}).
\end{equation*}
This finishes the proof of the corollary.
\end{proof}

Let $n \geq 2$. For $i \in \setseq$ define $\mset_i \subseteq \real^{n}$ by
\begin{equation}\label{defmset}
\mset_{i} =\sset_{i} \times \real^{n-2}.
\end{equation}

Let $\|\cdot\|$ denote the Euclidean norm on $\real^n$. We use $[x,y] \subseteq \real^n$ to denote a closed line segment, where $x,y \in \real^n$.

\begin{theorem}\label{resultofsection}
The family of subsets $\{\mset_i \subseteq \real^{n} \mid i \in \setseq\}$ satisfies the following three statements.
\begin{enumerate}
\item[(i)]
If $i \in \setseq$ then $\mset_i$ is non-empty, closed and has measure zero.
\item[(ii)]
If $i,j \in \setseq$ and $i \preceq j$ then $\mset_i \subseteq \mset_j$.
\item[(iii)]
If $i,j \in \setseq$ with $i \prec j$ and $\e > 0$, then there exists
$\alpha = \alpha(i,j,\e)>0$
such that whenever $\delta \in (0,\alpha)$,
$u_{1},u_{2},u_{3}$ are in the closed unit ball $D^n$ of $\ \real^n$
and $x\in\mset_i$,
there exist $u_{1}',u_{2}',u_{3}'\in\real^n$
with $\|u_{m}'-u_{m}\| \leq \e$ and
\begin{equation*}
[x+\delta u_{1}',x+\delta u_{3}'] \cup
[x+\delta u_{3}',x+\delta u_{2}']
\subseteq \mset_{j}.
\end{equation*}
\end{enumerate}
\end{theorem}

\begin{proof}
Recall that for each $i \in \setseq$,
$\sset_{i}$ is a non-empty closed set of measure zero and that
$\sset_{i}\subseteq \sset_{j}$ whenever $i \preceq j$. Hence
\eqref{defmset} implies (i) and (ii). For (iii), let $\alpha=\delta_4(i,j,\e)$ from
Corollary~\ref{corint}, part \eqref{corint-4} and $\delta\in(0,\alpha)$. Suppose $x \in \mset_i$ and $u_m \in D^{n}$, $m = 1,2,3$. Write $x = (x',y')$ and $u_{m} = (v_{m},h_{m})$ with $x' \in \sset_i$, $v_{m} \in D^2$ and $y',h_{m} \in \real^{n-2}$.

By Corollary~\ref{corint}, part \eqref{corint-4}, we can find $v_{1}',v_{2}',v_{3}' \in\real^{2}$
with $\|v_{m}'-v_{m}\| \leq \e$ and
\begin{equation*}
[x'+\delta v_{1}',x'+\delta v_{3}'] \cup
[x'+\delta v_{3}',x'+\delta  v_{2}'] \subseteq
\sset_{j}.
\end{equation*}
Then setting $u_{m}' = (v_{m}',h_{m})$ we have $\|u_{m}'-u_{m}\| =
\|v_{m}'-v_{m}\| \leq \e$ and
\begin{equation*}
[x+\delta u_{1}',x+\delta u_{3}'] \cup
[x+\delta u_{3}',x+\delta  u_{2}'] \subseteq
\mset_{j}.
\end{equation*}
\end{proof}

\section{A point with almost locally maximal directional derivative}
\label{sec3}
In this section we work on a general real Hilbert space $H$, although
eventually we shall only be concerned with the case in which $H$ is
finite dimensional. Let $\langle,\rangle$ denote the inner product on
$H$, $\|\cdot\|$ the norm and let
$S(H)$ denote the unit sphere of $H$. We shall assume that the family
$\{\mset_{i} \subseteq H \mid i \in\setseq \}$
consists of closed sets such that $\mset_{i} \subseteq \mset_{j}$ whenever
$i \preceq j$, where the index set $(\setseq,\preceq)$ is a dense,
chain complete poset.

For a Lipschitz function $h\colon  H \rightarrow \real$ we write $D^h$ for the set of all pairs
$(x,e)\in H\times S(H)$ such that
the directional derivative $h'(x,e)$ exists and, for each $i\in\setseq$,
we let $D^{h}_{i}$  be the set of all $(x,e)\in D^h$ such that
$x\in\mset_i$.
If, in addition, $h\colon H\to\real$ is linear then we write $\|h\|$ for the operator norm of $h$.

\begin{theorem}
\label{thincr}
Suppose $f_{0}\colon  H \rightarrow \real$ is a Lipschitz function, $i_{0} \in\setseq$, $(x_{0},e_{0}) \in D^{f_{0}}_{i_{0}}$, $\delta_{0},\mu,K > 0$ and $j_{0} \in\setseq$ with $i_{0} \prec j_{0}$. Then there
exists a Lipschitz function $f\colon  H \rightarrow \real$
such that $f-f_{0}$ is linear with norm not greater than $\mu$ and a pair $(x,e) \in D^{f}_{i}$, where
$\|x-x_{0}\| \leq \delta_{0}$ and $i \in (i_{0},j_{0})$,
such that the directional derivative $f'(x,e) > 0$ is almost locally maximal in the following sense. For any
$\e > 0$ there exists $\delta_{\e} > 0$ and $j_{\e} \in (i,j_{0})$ such that
whenever $(x',e') \in D^{f}_{j_{\e}}$ satisfies

\begin{enumerate}
\item[(i)]
$\|x'-x\| \leq \delta_{\e}$, $f'(x',e') \geq f'(x,e)$ and
\item[(ii)]
for any $t\in\real$
\begin{equation}
\label{eqincr}
|(f(x'+te)-f(x'))-(f(x+te)-f(x))| \leq K \sqrt{f'(x',e')-f'(x,e)}|t|,
\end{equation}
\end{enumerate}
then we have $f'(x',e') < f'(x,e) + \e$.
\end{theorem}

We devote the rest of this section to proving Theorem~\ref{thincr}.

Without loss of generality we may assume
$\mathrm{Lip}(f_{0}) \leq 1/2$ and $K \geq 4$. By replacing $e_{0}$ with $-e_{0}$ if necessary we may assume $f_{0}'(x_{0},e_{0}) \geq 0$.

If $h$ is a Lipschitz function, the pairs $(x,e)$, $(x',e')$ belong to $D^h$ and $\sigma\ge0$
we write
\begin{equation}\label{eqleq}
\pairorder xe{x'}{e'}h\sigma
\end{equation}
if
$h'(x,e) \leq h'(x',e')$ and
for all $t \in \real$,
\begin{align*}
|(h(x'+te)-h(x'))-(h(x+te)-h(x))| \leq K\Bigl(\sigma + \sqrt{h'(x',e')-h'(x,e)}\Bigr)|t|.
\end{align*}

We shall construct by recursion a sequence of Lipschitz functions
$f_{n}\colon  H \rightarrow \real$, sets $D_n \subseteq D^{f_{0}}$ and pairs $(x_n,e_n) \in D_n$ such that the directional derivative $f_n'(x_n,e_n)$ is within $\lambda_n$ of its supremum over $D_n$, where $\lambda_n > 0$. We shall show that $f=\lim f_n$ and $(x,e)=\lim(x_n,e_n)$ have the desired properties. The constants $\delta_m$ will be used to bound $\|x_n-x_m\|$ for $n \geq m$ whereas $\sigma_m$ will bound $\|e_n-e_m\|$ and $t_m$ will control $\|f_n-f_m\|$ for $n\geq m$.

The recursion starts with
$f_{0}$, $i_{0}$, $j_{0}$, $x_{0}$, $e_{0}$, $\delta_{0}$
defined in the statement of Theorem~\ref{thincr}.
Let $\sigma_{0} = 2$ and $t_{0} = \min(1/4,\mu/2)$.
For $n \geq 1$ we shall pick
\begin{equation*}
f_{n},\sigma_{n},t_{n},\lambda_{n},D_{n},x_{n},e_{n},\e_{n},i_{n},j_{n},\delta_{n}
\end{equation*}
in that order where
\begin{itemize}
\item
$i_{n},j_{n} \in\setseq$ with $i_{n-1}\prec i_n\prec j_n \prec j_{n-1}$,
\item
$D_{n}$ are non-empty subsets of $D^{f_{0}} \subseteq H\times S(H)$,
\item
$\sigma_{n},t_{n},\lambda_{n},\e_{n},\delta_{n} > 0$,
\item
$f_{n}\colon H\to \real$ are Lipschitz functions,
\item $(x_n,e_n)\in D_n$.
\end{itemize}
\begin{algorithm}\label{alg}
Given $n\ge1$ choose
\begin{enumerate}
\item[(1)]
\label{one} $f_{n}(x) = f_{n-1}(x) + t_{n-1}\langle x,e_{n-1}\rangle$,
\item[(2)]
\label{two} $\sigma_{n} \in (0,\sigma_{n-1}/4)$,
\item[(3)]
\label{three} $t_{n} \in (0,\min(t_{n-1}/2,\sigma_{n-1}/4n))$,
\item[(4)]
\label{four} $\lambda_{n} \in (0,t_{n} \sigma_{n}^{2}/2)$,
\item[(5)]
\label{five} $D_{n}$ to be the set of all pairs $(x,e)$ such that
$(x,e) \in D^{f_n}_{i}=D^{f_0}_i$ for some $i \in (i_{n-1},j_{n-1})$,
$\|x-x_{n-1}\| < \delta_{n-1}$ and
$$
\pairorder {x_{n-1}}{e_{n-1}}{x}{e}{f_{n}}{\sigma_{n-1}-\e}
$$
for some $\e \in( 0,\sigma_{n-1})$,
\item[(6)]
\label{six} $(x_{n},e_{n}) \in D_{n}$ such that
$f_{n}'(x,e)
\leq
f_n'(x_n,e_n)+\lambda_n$  for every $(x,e)\in D_n$,
\item[(7)]
$\e_{n} \in (0,\sigma_{n-1})$ such that
$\pairorder {x_{n-1}}{e_{n-1}}{x_n}{e_n}{f_n}{\sigma_{n-1}-\e_n}$,
\item[(8)]
$i_{n} \in (i_{n-1},j_{n-1})$ such that $x_{n} \in \mset_{i_{n}}$,
\item[(9)]
$j_{n} \in (i_{n},j_{n-1})$ and
\item[(10)]
\label{ten} $\delta_{n} \in (0,(\delta_{n-1} - \|x_{n}-x_{n-1}\|)/2)$
such that
for all $t$ with $|t| < \delta_{n}/\e_{n}$
\begin{align}\label{eqfnen-1}
|(f_{n}(x_{n}+te_{n})-&f_{n}(x_{n}))-(f_{n}(x_{n-1}+te_{n-1})-f_{n}(x_{n-1}))|\\
\leq
(&f_{n}'(x_{n},e_{n})-f_{n}'(x_{n-1},e_{n-1})+\sigma_{n-1})|t|.\notag
\end{align}
\end{enumerate}
\end{algorithm}

Note that (5) implies that $(x_{n-1},e_{n-1})\in D_n$,
and so $D_n \neq \emptyset$; further as $f_n$ is Lipschitz we see $\sup_{(x,e) \in D_n} f_n'(x,e) < \infty$. Therefore
we are able to pick $(x_n,e_n)\in D_n$ with the property of (6).

The definition (5) of $D_n$ then implies that $\e_n$ and $i_n$ exist with the properties of (7)--(8). Further, we have $\|x_n-x_{n-1}\| < \delta_{n-1}$ and
\begin{equation}\label{eqfnder}
f_{n}'(x_{n},e_{n}) \geq f_{n}'(x_{n-1},e_{n-1}).
\end{equation}
These allow us to choose $\delta_n$ as in (10).

Observe that the positive sequences $\sigma_n$, $t_n$, $\lambda_n$,
$\delta_n$, $\e_n$ all tend to $0$:
$\sigma_n\in(0,\sigma_{n-1}/4)$ by (2),
$t_n\in(0,t_{n-1}/2)$ by (3),
$\lambda_n\in(0,t_n\sigma_n^2/2)$ by (4),
$\delta_n\in(0,\delta_{n-1}/2)$ by (10) and
$\e_n\in(0,\sigma_{n-1})$ by (7).
 Further from (10),
\begin{equation}\label{ballsub}
\overline{B(x_{n},\delta_{n})} \subseteq B(x_{n-1},\delta_{n-1}).
\end{equation}

Note that (1)
and (3)
imply
$f_n(x)=f_0(x)+\langle x,\sum_{k=0}^{n-1} t_k e_k\rangle$ and, as the
Lipschitz constant
$\mathrm{Lip}(f_{0}) \leq \frac{1}{2}$, $t_{k+1} \leq t_{k}/2$ and
$t_{0} \leq \frac{1}{4}$, we deduce that $\mathrm{Lip}(f_{n}) \leq 1$
for all $n$.

Let $\e_{n}' > 0$ be given by
\begin{equation}\label{defepsn'}
\e_{n}' = \min(\e_{n}/2,\sigma_{n-1}/4).
\end{equation}
\begin{lemma}
\label{lemleq}
The following three statements hold.
\begin{enumerate}
\item[(i)]
If $n\ge1$ and $(x,e) \in D_{n+1}$, then
$$
\pairorder {x_{n-1}}{e_{n-1}}{x}{e}{f_n}{\sigma_{n-1}-\e'_n}.
$$
\item[(ii)]
If $n\geq1$ then $D_{n+1} \subseteq D_{n}$.
\item[(iii)]
If $n\ge0$ and $(x,e) \in D_{n+1}$, then
$\|e-e_{n}\| \leq \sigma_{n}$.
\end{enumerate}
\end{lemma}
\begin{proof}
For $n=0$, condition
(iii) is satisfied as $\sigma_0=2$.
Now it is enough to check that if $n\geq1$ and the condition (iii) is satisfied
for $n-1$, then
conditions (i)--(iii) are satisfied for $n$.
The Lemma then will follow by induction.

Assume $n\ge1$ and $\|e'-e_{n-1}\| \leq\sigma_{n-1}$
for all $(x',e') \in D_{n}$.
Then we have
\begin{equation}\label{enen-1}
\|e_{n}-e_{n-1}\|
\leq
\sigma_{n-1}
\end{equation}
as $(x_{n},e_{n}) \in D_{n}$. Now fix any $(x,e) \in D_{n+1}$.
Using (1)
and (5)
of Algorithm~\ref{alg} and $\langle e,e_n\rangle \leq 1$ we get
\begin{align}\label{eqfn-fn+1}
A&:=f_{n}'(x,e)-f_{n}'(x_{n},e_{n})\\
& = f_{n+1}'(x,e)-t_n \langle e,e_n \rangle - f_{n+1}'(x_n,e_n) + t_n
\notag
\\
& \geq
f_{n+1}'(x,e) - f_{n+1}'(x_{n},e_{n})
\geq 0,\notag
\end{align}
so that
$$
f_{n}'(x,e) \geq f_{n}'(x_{n},e_{n}) \geq f_{n}'(x_{n-1},e_{n-1})
$$
by \eqref{eqfnder}.
If we let
$B =f_{n}'(x,e)-f_{n}'(x_{n-1},e_{n-1})$ we have
\begin{equation*}
K(\sqrt{B} - \sqrt{A})
\geq B-A
= f_{n}'(x_{n},e_{n})-f_{n}'(x_{n-1},e_{n-1}),
\end{equation*}
since $K \geq 4$ and $0 \leq A \leq B \leq 2$, using $\mathrm{Lip}(f_n) \leq 1$ in the final inequality.
Together with \eqref{eqfn-fn+1} this implies that
\begin{equation}\label{ksqrtb}
(f_{n}'(x_{n},e_{n})-f_{n}'(x_{n-1},e_{n-1})) +
K\sqrt{f_{n+1}'(x,e)-f_{n+1}'(x_{n},e_{n})}
\leq
K\sqrt B.
\end{equation}

In order to prove (i), we need to establish an upper estimate for
\begin{equation}\label{eqii}
|(f_{n}(x+te_{n-1})-f_{n}(x))-(f_{n}(x_{n-1}+te_{n-1})-f_{n}(x_{n-1}))|.
\end{equation}
For every $|t|<\delta_n/\e_n$, using
\begin{align*}
|(f_n(x+te_n)&-f_n(x))-(f_{n}(x_n+te_{n})-f_{n}(x_n))|\\
&=
|( f_{n+1}(x+te_n)-f_{n+1}(x))-(f_{n+1}(x_n+te_{n})-f_{n+1}(x_n))|\\
&\leq
K \Bigl(\sigma_{n} + \sqrt{f_{n+1}'(x,e)-f_{n+1}'(x_{n},e_{n})}\Bigr)|t|
\end{align*}
and \eqref{eqfnen-1}, we get from \eqref{ksqrtb}
\begin{align*}
|(f_{n}(x&+te_{n-1})-f_{n}(x))-(f_{n}(x_{n-1}+te_{n-1})-f_{n}(x_{n-1}))|\\
&\le \sigma_{n-1}|t| + K\Bigl(\sigma_{n} +
\sqrt{f_{n}'(x,e)-f_{n}'(x_{n-1},e_{n-1})}\Bigr)|t| + \|e_{n}-e_{n-1}\|\cdot|t|.
\end{align*}
Using \eqref{enen-1} and $K\ge4$ we see that the latter does not exceed
\begin{align*}
K\Bigl(\sigma_{n-1}/2 + \sigma_{n} +
&\sqrt{f_{n}'(x,e)-f_{n}'(x_{n-1},e_{n-1})}\Bigr)|t|\\
\leq
K\Bigl(\sigma_{n-1}-\e_{n}' +
&\sqrt{f_{n}'(x,e)-f_{n}'(x_{n-1},e_{n-1})}\Bigr)|t|
\end{align*}
as $\sigma_n\le\sigma_{n-1}/4$ by (2)
of Algorithm~\ref{alg}
and $\e_n'\leq\sigma_{n-1}/4$ by \eqref{defepsn'}.

Now we consider the case $|t| \geq \delta_{n}/\e_{n}$.
We have from (7)
of Algorithm~\ref{alg} that
$$\pairorder {x_{n-1}}{e_{n-1}}{x_{n}}{e_{n}}{f_{n}}{\sigma_{n-1}-\e_{n}}.$$
Using
this together with
\begin{align*}
\max\Bigl\{
|f_n(&x)-f_n(x_n)|,
|f_n(x+te_{n-1})-f_n(x_n+te_{n-1})|\Bigr\} \\
&\le
\|x-x_n\|
\le
\delta_n
\le
\e_n|t|
\le
K\e_n|t|/4
\end{align*}
we get
\begin{align*}
|(f_{n}(x&+te_{n-1})-f_{n}(x))-(f_{n}(x_{n-1}+te_{n-1})-f_{n}(x_{n-1}))|\\
&\leq
K\Bigl(\sigma_{n-1}-\e_{n}/2 + \sqrt{f_{n}'(x_{n},e_{n}) -
  f_{n}'(x_{n-1},e_{n-1})}\Bigr)|t|\\
&\leq
K\Bigl(\sigma_{n-1}-\e_{n}' + \sqrt{f_{n}'(x,e) -
  f_{n}'(x_{n-1},e_{n-1})}\Bigr)|t|
\end{align*}
because $f_{n}'(x_{n},e_{n}) \leq f_{n}'(x,e)$ from \eqref{eqfn-fn+1}.
Thus (i) is proved.

Further, for $(x,e) \in D_{n+1}$ we have $x \in B(x_{n},\delta_{n}) \subseteq B(x_{n-1},\delta_{n-1})$, using \eqref{ballsub}, and
$x \in \mset_{i}$ where
$$
i \in
(i_{n+1},j_{n+1}) \subseteq (i_{n},j_{n}).
$$
Hence $(x,e) \in D_{n}$ follows from (i). This establishes (ii).

Finally to see (iii),
let $(x,e) \in D_{n+1}$ and recall that (5)
of Algorithm~\ref{alg}
implies $f_{n+1}'(x_{n},e_{n}) \leq f_{n+1}'(x,e)$. By (1)
of Algorithm~\ref{alg}, this
can be written
\begin{equation*}f_n'(x_n,e_n) + t_n \langle e_n,e_n\rangle
\le f_n'(x,e) + t_n \langle e,e_n\rangle.
\end{equation*}
Since $(x,e) \in D_{n}$ by (ii), we have
$f_{n}'(x,e) \leq f_{n}'(x_n,e_n) + \lambda_{n}$. Combining the two inequalities we get
$t_{n}\leq t_{n} \langle e,e_{n}\rangle + \lambda_{n}$.
Hence $\langle e,e_{n}\rangle \geq 1 - \lambda_{n}/t_{n}$ so that
\begin{equation*}
\|e-e_n\|^{2} =
2-2\langle e,e_{n}\rangle \leq 2\lambda_{n}/t_{n} \leq
\sigma_{n}^{2}
\end{equation*}
using (4)
of Algorithm~\ref{alg}.

This completes the proof of the lemma.
\end{proof}

We now show that the sequences $x_{n}$, $e_{n}$ and $f_{n}$ converge and establish some properties of their limits.

Recall first that $i_{n-1} \prec i_{n} \prec j_{n} \prec j_{n-1}$ for all
$n \geq 1$. The set $\{i_{n}\mid n \in \mathbb{N}\}$ is thus a non-empty chain in $\setseq$. Therefore, it has a
supremum $i\in\setseq$. Further, as $i_{n} \in (i_{m+1},j_{m+1})$ for $n \geq m+2$, we know $i \in [i_{m+1},j_{m+1}] \subseteq (i_{m},j_{m})$ for all $m$.

\begin{lemma}
\label{limitproperties}

We have $x_{m} \rightarrow x$, $e_{m} \rightarrow e$ and
$f_{m} \rightarrow f$ where
\begin{enumerate}
\item[(i)]
$f\colon  H \rightarrow \real$ is a Lipschitz function with $\mathrm{Lip}(f) \leq 1$,
\item[(ii)]
$f-f_{m}$ is linear and $\|f-f_{m}\| \leq 2t_{m}$ for all $m$,
\item[(iii)]
$x \in \mset_{i}$, $\|x-x_{m}\| < \delta_{m}$ and
$\|e-e_{m}\| \leq \sigma_{m}$,
\item[(iv)]
$f'(x,e)$ exists, is positive and $f_{m}'(x_{m},e_{m}) \nearrow f'(x,e)$,
\item[(v)]
$\pairorder {x_{m-1}}{e_{m-1}}{x}{e}{f_m}{\sigma_{m-1}-\e'_m}$ and
\item[(vi)]
$(x,e) \in D_{m}$ for all $m$.
\end{enumerate}
\end{lemma}

\begin{proof}
Letting $f(x)=f_0(x)+\langle x,\sum_{k\ge0} t_k e_k\rangle$ we deduce
$f_{n} \rightarrow f$ and (i), (ii) from
$f_n(x)=f_0(x)+\langle x,\sum_{k=0}^{n-1} t_k e_k\rangle$,
$\mathrm{Lip}(f_{n}) \leq 1$ and $t_{n+1} \leq t_{n}/2$.

For $n \geq m$, by parts  (ii) and (iii) of Lemma~\ref{lemleq} we have
$(x_{n},e_{n}) \in D_{n+1} \subseteq D_{m+1}$
and $\|e_{n}-e_{m}\| \leq \sigma_{m}$.
The former implies
$\|x_{n}-x_{m}\| < \delta_{m}$ by the definition of
$D_{m+1}$.
As $\delta_{m}$ and $\sigma_m$ tend to $0$, the sequences $(x_{n})$ and $(e_{n})$
are Cauchy so that they converge to some $x \in H$ and $e \in S(H)$ respectively.
Taking the $n \to \infty$ limit we obtain $\|x-x_{m}\| \leq \delta_{m}$
and $\|e-e_{m}\| \leq \sigma_{m}$. The former implies
$x \in \overline{B(x_{m},\delta_{m})} \subseteq
B(x_{m-1},\delta_{m-1})$
for all $m\ge1$, using \eqref{ballsub}.

To complete (iii), note that from (8) of Algorithm~\ref{alg} we have $x_{n} \in \mset_{i_{n}} \subseteq \mset_{i}$ for all $n$, as $i_{n} \preceq i$. Now $x_{n} \rightarrow x$ and $\mset_{i}$ is closed so that $x \in \mset_{i}$.

We now show that the directional derivative derivative $f'(x,e)$ exists.

For $n\geq m$ we have $(x_n,e_n)\in D_{m+1}$; therefore
by part (i) of
Lemma~\ref{lemleq} we know
\begin{equation}\label{uniformeps}
\pairorder {x_{m-1}}{e_{m-1}}{x_n}{e_n}{f_m}{\sigma_{m-1}-\e'_m}.
\end{equation}
Now the sequence
$\left(f_{n}'(x_{n},e_{n})\right)$ is strictly increasing and is
non-negative as $f_{0}'(x_{0},e_{0}) \geq 0$ and $f_{n}'(x_{n},e_{n})
< f_{n+1}'(x_{n},e_{n}) \leq f_{n+1}'(x_{n+1},e_{n+1})$.
It is bounded above by $\mathrm{Lip}(f_{n}) \leq 1$ so that it
converges to some $L \in (0,1]$.
As $\|f-f_{n}\| \rightarrow 0$ we also have $f'(x_{n},e_{n}) \rightarrow L$ and $f_{n+1}'(x_{n},e_{n}) \rightarrow L$. Note then that for each fixed $m$,
$$
f_{m}'(x_{n},e_{n})-f_{m}'(x_{m-1},e_{m-1})
\xrightarrow[n\to\infty]{}
s_{m},
$$
where
\begin{equation}\label{defsm}
s_{m}
=
(f_{m}-f)(e) +
L - f_{m}'(x_{m-1},e_{m-1})
\xrightarrow[m\to\infty]{} 0.
\end{equation}
As $f_{m}'(x_{n},e_{n}) \geq f_{m}'(x_{m-1},e_{m-1})$ from \eqref{uniformeps} we have
$s_{m} \geq 0$ for each $m$.
Taking $n \rightarrow \infty$ in \eqref{uniformeps} we thus obtain
\begin{equation}\label{eqfm}
|(f_{m}(x+te_{m-1})-f_{m}(x))-(f_{m}(x_{m-1}+te_{m-1})-f_{m}(x_{m-1}))|
\leq r_{m}|t|
\end{equation}
for any $t \in \real$,
where
\begin{equation}
\label{defrm}
r_{m}= K(\sigma_{m-1} - \e_{m}' + \sqrt{s_{m}})\rightarrow 0.
\end{equation}
Using $\|f-f_m\| \leq 2t_m$, $\|e-e_{m-1}\| \leq \sigma_{m-1}$ and $\mathrm{Lip}(f) \leq 1$:
\begin{equation}\label{eqfmmod}
|(f(x+te)-f(x))-(f_{m}(x_{m-1}+te_{m-1})-f_{m}(x_{m-1}))| \leq (r_m+2t_m+\sigma_{m-1})|t|.
\end{equation}
Let $\e > 0$. Pick $m$ such that
\begin{equation}\label{mlarge}
r_m+2t_m+\sigma_{m-1} \leq \e/3 \textrm{ and }|f_{m}'(x_{m-1},e_{m-1})-L| \leq \e/3
\end{equation}
and $\delta > 0$ with
\begin{equation}\label{dirderivatm}
|f_{m}(x_{m-1}+te_{m-1})-f_{m}(x_{m-1})-f_{m}'(x_{m-1},e_{m-1})t| \leq \e |t|/3
\end{equation}
for all $t$ with $|t| \leq \delta$. Combining \eqref{eqfmmod},
\eqref{mlarge} and \eqref{dirderivatm} we obtain
\begin{equation*}
|f(x+te)-f(x)-Lt| \leq \e |t|
\end{equation*}
if $|t| \leq \delta$. Hence the directional derivative
$f'(x,e)$ exists and equals $L$.
As $L > 0$
and
$f_{n}'(x_{n},e_{n})$ is an increasing sequence that tends to $L$,
we get (iv).

Note further that, as $f_m-f$ is linear, the directional derivative $f_m'(x,e)$ also exists and equals $(f_m-f)(e)+L$. Hence from \eqref{defsm}
\begin{equation*}
s_{m} = f_{m}'(x,e)-f_{m}'(x_{m-1},e_{m-1}).
\end{equation*}
As $s_{m} \geq 0$ for all $m$, we conclude that $f_{m}'(x,e) \geq
f_{m}'(x_{m-1},e_{m-1})$ for all $m$.
Further from \eqref{eqfm} and \eqref{defrm},
\begin{align*}
|(f_{m}(x&+te_{m-1})-f_{m}(x))-(f_{m}(x_{m-1}+te_{m-1})-f_{m}(x_{m-1}))|
\\
&\leq
K\Bigl(\sigma_{m-1} - \e_{m}' +
\sqrt{f_{m}'(x,e)-f_{m}'(x_{m-1},e_{m-1})}\Bigr)|t|
\end{align*}
for any $t$. Hence
$$
\pairorder {x_{m-1}}{e_{m-1}}{x}{e}{f_m}{\sigma_{m-1}-\e'_m}.
$$

This establishes (v). Finally (vi) follows immediately from
(iii), (iv), (v) and the fact $i \in (i_{m},j_{m})$.
\end{proof}

\noindent\textit{Proof of Theorem~\ref{thincr}.}
From Lemma~\ref{limitproperties}~(i)--(ii)
the Lipschitz function
$f\colon H\to\real$ is such that $f-f_0$ is linear and
$\|f-f_0\|\le 2t_0\le\mu$.
Recall that $i \in (i_{m},j_{m})$ for all $m$;
in particular $i \in (i_{0},j_{0})$.
By parts (iii) and (iv) of Lemma~\ref{limitproperties} we see that
$(x,e) \in D^{f}_{i}$ and $f'(x,e)>0$.

We are left needing to verify that
the directional derivative $f'(x,e)$ is almost locally maximal in the sense of Theorem~\ref{thincr}.
\begin{lemma}
If $\e > 0$ then there exists
$\delta_{\e} > 0$ and $j_{\e} \in (i,j_{0})$ such that whenever
\begin{equation*}
\pairorder {x}{e}{x'}{e'}{f}{0}
\end{equation*}
with $\|x'-x\| \leq \delta_{\e}$ and $x' \in
\mset_{j_{\e}}$, we have $f'(x',e') < f'(x,e) + \e$.
\end{lemma}
\begin{proof}
Pick $n$ such that
\begin{equation}\label{lambdantn}
n \geq 4/\sqrt{\e}
\textrm{ and }
\lambda_{n},t_{n} \leq
\e/4.
\end{equation}
Let $j_{\e} = j_{n} \in (i,j_{0})$. Find
$\delta_{\e} > 0$ such that
\begin{equation}\label{delta}
\delta_{\e} < \delta_{n-1}-\|x-x_{n-1}\|
\end{equation}
and
\begin{align}\label{defdeleps}
|(f_{n}(x+te)&-f_{n}(x))-(f_{n}(x_{n-1}+te_{n-1})-f_{n}(x_{n-1}))|
\\
& \leq (f_{n}'(x,e)-f_{n}'(x_{n-1},e_{n-1}) + \sigma_{n-1}) |t|\notag
\end{align}
for all $t$ with $|t| < \delta_{\e}/\e_{n}'$, where $\e_n'$ is given
by \eqref{defepsn'}.
Lemma~\ref{limitproperties}~(iii) and the fact that
$f_{n}'(x,e)-f_{n}'(x_{n-1},e_{n-1}) \geq 0$ from
Lemma~\ref{limitproperties}~(v) guarantee the existence of such
$\delta_\e$.

Now suppose that
\begin{equation}\label{choicex'}
\begin{cases}
\pairorder xe{x'}{e'}f0,\\
\|x'-x\| \leq \delta_{\e}\textrm{ and }x' \in \mset_{j_{\e}},
\\
f'(x',e') \geq f'(x,e) + \e.
\end{cases}
\end{equation}
We aim to show that $(x',e') \in D_{n}$. That will lead to a
contradiction
since, together with (6)
in Algorithm~\ref{alg} and Lemma~\ref{limitproperties}~(iv), this
would imply
$$
f_{n}'(x',e') \leq f_{n}'(x_{n},e_{n}) +
\lambda_{n} \leq  f'(x,e) + \lambda_{n}
$$
so that
\begin{equation*}
f'(x',e') \leq f'(x,e) + \lambda_n + 2t_n,
\end{equation*}
by Lemma~\ref{limitproperties}~(ii). This contradicts \eqref{lambdantn} and \eqref{choicex'}.

Since \eqref{delta} and \eqref{choicex'}
imply $x' \in
B(x_{n-1},\delta_{n-1})$ and $x' \in \mset_{j_{\e}}$ with $j_{\e} =j_{n}\in
 (i_{n-1},j_{n-1})$, to prove $(x',e') \in
D_{n}$ it is enough to show that
\begin{equation}\label{eqlast}
\pairorder {x_{n-1}}{e_{n-1}}{x'}{e'}{f_{n}}{\sigma_{n-1}-\e_{n}'/2};
\end{equation}
see (5)
in Algorithm~\ref{alg}.

First, note that $f_{n}'(x',e')-f_{n}'(x,e) \geq
f'(x',e')-f'(x,e)-2\|f_{n}-f\| \geq \e - 4t_{n} \geq 0$, so that $f_{n}'(x',e') \geq f_{n}'(x,e) \geq
f_{n}'(x_{n-1},e_{n-1})$.

Let $A = f'(x',e')-f'(x,e)$ and $B =
f_{n}'(x',e')-f_{n}'(x,e)$. We have $A \geq \e$ and $B \geq 0$;
therefore by (3)
of
Algorithm~\ref{alg}, Lemma~\ref{limitproperties} (ii) and
\eqref{lambdantn}
$$
\sqrt{A}-\sqrt{B}
\leq
\frac{A-B}{\sqrt{\e}}
=
\frac{(f-f_{n})(e'-e)}{\sqrt{\e}}
\leq
\frac{4t_{n}}{\sqrt{\e}}
\leq
nt_{n}
\leq
\sigma_{n-1}/4.
$$
Further, let $C = f_{n}'(x',e')-f_{n}'(x_{n-1},e_{n-1})$.
Since $f_{n}'(x_{n-1},e_{n-1}) \leq f_{n}'(x,e)$ and the Lipschitz
constant $\text{Lip}(f_n)$ does not exceed $1$, we have
$0 \le B \leq C \leq 2$, so that
$$
K\sqrt{C}-K\sqrt{B}
\geq
C-B
= f_{n}'(x,e)-f_{n}'(x_{n-1},e_{n-1})
$$
as $K \geq 4$. Hence
\begin{align}\label{fne'}
(f_{n}'(x,e)&-f_{n}'(x_{n-1},e_{n-1})) + K\sqrt{f'(x',e')-f'(x,e)}\notag
\\
&\leq
K\sqrt{C}-K\sqrt{B} + K(\sqrt{B}+\sigma_{n-1}/4)
\\ \notag
&=
K (\sqrt{f_{n}'(x',e')-f_{n}'(x_{n-1},e_{n-1})} + \sigma_{n-1}/4).
\end{align}

In order to check \eqref{eqlast}, we need to obtain an upper estimate
for
\begin{equation}\label{eqe'}
|(f_{n}(x'+te_{n-1})-f_{n}(x'))-(f_{n}(x_{n-1}+te_{n-1})-f_{n}(x_{n-1}))|.
\end{equation}
If $|t| < \delta_{\e}/\e_{n}'$, we can use
\begin{align*}
|(f_{n}(&x'+te)-f_{n}(x'))-(f_{n}(x+te)-f_{n}(x))|\\
&= |(f(x'+te)-f(x'))-(f(x+te)-f(x))| \le
K\sqrt{f'(x',e')-f'(x,e)}|t|
\end{align*}
and \eqref{defdeleps}
to deduce that \eqref{eqe'} is no greater than
\begin{multline*}
(f_{n}'(x,e)-f_{n}'(x_{n-1},e_{n-1})+\sigma_{n-1})|t|\\
+
K\sqrt{f'(x',e')-f'(x,e)}|t|
+
\|e-e_{n-1}\|\cdot|t|
\end{multline*}
since $\text{Lip}(f_n) \leq 1$. Using \eqref{fne'}, $\|e-e_{n-1}\| \leq \sigma_{n-1}$, $\e_n' \leq \sigma_{n-1}/4$ and $K\ge4$
we get that the latter does not exceed
$$
K \left(\sigma_{n-1} - \e_{n}'/2 +
\sqrt{f_{n}'(x',e')-f_{n}'(x_{n-1},e_{n-1})}\right)|t|.
$$
On the other hand, for $|t| \geq \delta_{\e}/\e_{n}'$ we have
$2\|x-x'\| \leq 2\e_{n}'|t| \leq K\e_{n}'|t|/2$ so,
using this together with Lemma~\ref{limitproperties}~(v), $\text{Lip}(f_n) \leq 1$ and $f_{n}'(x,e) \leq f_{n}'(x',e')$, we get
\begin{align*}
|(f_{n}&(x'+te_{n-1})-f_{n}(x'))-(f_{n}(x_{n-1}+te_{n-1})-f_{n}(x_{n-1}))|
\\
&\leq
2\|x'-x\|
+
K\Bigl(\sigma_{n-1} - \e_{n}' +
\sqrt{f_{n}'(x,e)-f_{n}'(x_{n-1},e_{n-1})}\Bigr)|t|
\\
&\leq
K\Bigl(\sigma_{n-1} - \e_{n}'/2 +
\sqrt{f_{n}'(x',e')-f_{n}'(x_{n-1},e_{n-1})}\Bigr)|t|.
\end{align*}
Hence
$$
\pairorder {x_{n-1}}{e_{n-1}}{x'}{e'}{f_n}{\sigma_{n-1}-\e_{n}'/2}
$$
and we are done.
\end{proof}

This finishes the proof of Theorem~\ref{thincr}.

\section{A differentiability lemma}
\label{sec4}

As in the previous section, we shall mostly work on a real Hilbert space $H$,
though our eventual application will only use the case in which $H$ is
finite dimensional. Lemma~\ref{lemmax} is proved in general
real Banach space $X$.
Given $x,y$ in a linear space we use $[x,y]$ to denote the
closed line segment with endpoints $x$ and $y$.

We start by quoting Lemma~\ref{lempreiss}, which is \cite[Lemma 3.4]{P}.
This lemma can be understood as an improvement of the standard mean
value theorem applied to the function
$$
h(t)=\phi(t) - t\frac{\psi(s) - \psi(-s)}{2s} - \frac{\psi(s) +
  \psi(-s)}{2}.
$$
Roughly speaking, this ``generalised'' mean value theorem says that if
$h(s)=h(-s)=0$ and $h(\xi)\ne0$
then there is a
point $\tau\in[-s,s]$ such that the
derivative $h'(\tau)$ is bounded away from zero
by a term proportional to $|h(\xi)|/s$ and \eqref{phipsilempreiss}
holds. The latter inequality
essentially comes from the upper
bound for the slope $|h(\tau+t)-h(\tau)|/|t|$ by $(\mathbb M h')(\tau)$, where $\mathbb M$ is the Hardy-Littlewood
maximal operator.

We use this statement in order to show in Lemma~\ref{lemmax}  and Lemma~\ref{lemdiff} that if $f'(x,e)$ exists and is
maximal up to $\e$ among all directional derivatives of $f$ satisfying~\eqref{xx'}, at points in a $\delta_{\e}$-neighbourhood of $x$,
then $f$ is
\frechet{} differentiable at $x$. Lemma~\ref{lemmax}, which follows from
Lemma~\ref{lempreiss}, guarantees that if there is a direction $u$ in which $f(x+ru)-f(x)$ is not well approximated by $f'(x,e)\langle u,e\rangle$ then
we can find
a nearby point and direction $(x',e')$, satisfying the constraint~\eqref{xx'}, at which the directional derivative $f'(x',e')$ is at least as large as $f'(x,e)+\e$, a contradiction.

\begin{lemma}\label{lempreiss}
Suppose that $|\xi| < s < \rho$, $0 < \nu < \frac{1}{32}$, $\sigma >
0$ and $L > 0$ are real numbers and that $\phi$ and $\psi$ are
Lipschitz functions defined on the real line such that
$\mathrm{Lip}(\phi) + \mathrm{Lip}(\psi) \leq L$, $\phi(t) = \psi(t)$ for
$|t| \geq s$ and $\phi(\xi) \neq \psi(\xi)$. Suppose, moreover, that
$\psi'(0)$ exists and that
$$
|\psi(t)-\psi(0)-t\psi'(0)| \leq \sigma L |t|
$$
whenever $|t| \leq \rho$,
$$
\rho \geq s\sqrt{(sL)/(\nu|\phi(\xi)-\psi(\xi)|)},
$$
and
$$
\sigma \leq \nu^{3}\left(\frac{\phi(\xi)-\psi(\xi)}{sL}\right)^{2}.
$$
Then there is a $\tau \in (-s,s)\setminus\{\xi\}$
such that $\phi'(\tau)$ exists,
$$
\phi'(\tau) \geq \psi'(0) + \nu|\phi(\xi)-\psi(\xi)|/s,
$$
and
\begin{equation}\label{phipsilempreiss}
|(\phi(\tau + t)-\phi(\tau))-(\psi(t)-\psi(0))|
\leq
4(1+20\nu)\sqrt{[\phi'(\tau)-\psi'(0)]L}|t|
\end{equation}
for every $t\in\real$.
\end{lemma}

\begin{lemma}\label{lemmax}
Let $(X,\|\cdot\|)$ be a real Banach space,
$f\colon  X \rightarrow \mathbb{R}$ be a Lipschitz function with Lipschitz
constant $\mathrm{Lip}(f) > 0$ and let $\e\in(0,\mathrm{Lip}(f)/9)$.
Suppose $x\in X$, $e\in S(X)$ and $s > 0$ are such that
the directional derivative $f'(x,e)$ exists, is non-negative and
\begin{equation}
\label{eq80lip}
|f(x+te)-f(x)-f'(x,e)t| \leq \frac{\e^{2}}{160\mathrm{Lip}(f)} |t|
\end{equation}
for $|t| \leq s\sqrt{\frac{2\mathrm{Lip}(f)}{\e}}$.
Suppose further $\xi\in(-s/2,s/2)$
and $\lambda \in X$ satisfy
\begin{align}
\label{est160}&|f(x+\lambda)-f(x+\xi e)| \geq 240\e s,\\
\label{lambdaxi}&\|\lambda-\xi e\| \leq s\sqrt{\frac{\e}{\mathrm{Lip}(f)} }\\
\label{estup}\text{and \quad}&\frac{\|\pi se+\lambda\|}{|\pi s+\xi|} \leq 1 + \frac{\e}{4\mathrm{Lip}(f)}
\end{align}
for $\pi = \pm 1$. Then if $s_{1},s_{2},\lambda' \in X$ are such that
\begin{equation}\label{ests1s2}
\max(\|s_{1}-se\|,\|s_{2}-se\|) \leq \frac{\e^{2}}{320\mathrm{Lip}(f)^{2}}s
\end{equation}
and
\begin{equation}\label{estlam}
\|\lambda'-\lambda\| \leq \frac{\e s}{16\mathrm{Lip}(f)},
\end{equation}
we can find
$x' \in [x-s_1,x+\lambda'] \cup [x+\lambda',x+s_2]$ and $e' \in S(X)$ such
that the directional derivative $f'(x',e')$ exists,
\begin{equation}\label{derbound}
f'(x',e') \geq f'(x,e) + \e
\end{equation}
and for all $t \in \mathbb{R}$ we have
\begin{align}\label{estff'}
&|(f(x'+te)-f(x'))-(f(x+te)-f(x))|\\
&\notag\leq
25\sqrt{(f'(x',e')-f'(x,e))\mathrm{Lip}(f)}|t|.
\end{align}
\end{lemma}
\begin{proof}
Define constants
$L = 4\mathrm{Lip}(f)$, $\nu = \frac{1}{80}$, $\sigma =
\frac{\e^{2}}{20L^{2}}$ and $\rho = s \sqrt{\frac{L}{2\e}}$.
Let
\begin{equation}
\label{defphipsi}
\psi(t) = f(h(t))
\textrm{  and  }
\phi(t) = f(g(t)),
\end{equation}
where $h\colon  \real \rightarrow X$ is a mapping that is affine
on each of the intervals
$(-\infty,-s/2]$ and $[s/2,\infty)$ with $h(t) = x+te$ for
$t \in [-s/2,s/2]$ and $h(-s) = x-s_{1}$, $h(s) = x+s_{2}$ while
$g\colon  \real \rightarrow X$ is
a mapping that is affine on $[-s,\xi]$ and on $[\xi,s]$ with
$g(\xi) = x+\lambda'$ and $g(t) = h(t)$  for  $|t| \geq s$.

A simple calculation shows that \eqref{ests1s2} implies
\begin{equation}
\label{h't-e}
\|h'(t)-e\| \leq 2\frac{\max(\|s_{1}-se\|,\|s_{2}-se\|)}{s}
\leq
\frac{\e^{2}}{160\mathrm{Lip}(f)^{2}}
\end{equation}
for $t \in \real \setminus \{-s/2,s/2\}$.

Now the derivative of $g$ is given by
\begin{equation}\label{g't}
g'(t)=
\begin{cases}
(\lambda'+s_{1})/(\xi+s) &\textrm{ for }t \in (-s,\xi),\\
(\lambda'-s_{2})/(\xi-s)&\textrm{ for }t \in (\xi,s).\\
\end{cases}
\end{equation}

For $t \in (-s,\xi)$,
\begin{align*}
\left\|g'(t)-\frac{\lambda+se}{\xi+s}\right\|
&\leq
2\frac{\|\lambda'-\lambda\| + \|s_{1}-se\|}{s}\\
&
\leq
\frac{\e}{8\mathrm{Lip}(f)} + \frac{\e^{2}}{160\mathrm{Lip}(f)^{2}}
\leq
\frac{\e}{4\mathrm{Lip}(f)}
\end{align*}
using $|\xi| < s/2$, \eqref{ests1s2}, \eqref{estlam} and
$\e\leq \mathrm{Lip}(f)$. Hence
\begin{equation}
\label{speed1}
\|g'(t)\| \leq 1 + \frac{\e}{2\mathrm{Lip}(f)}
\end{equation}
and
\begin{align}\label{gest}
\|g'(t)-e\| \leq 3\sqrt{\frac{\e}{\mathrm{Lip}(f)}}.
\end{align}
The former follows from \eqref{estup} and the latter from
\begin{equation*}
\left\|\frac{\lambda+se}{\xi+s} - e\right\| = \left\| \frac{\lambda-\xi e}{\xi+s}\right\| \leq 2\frac{\|\lambda-\xi e\|}{s} \leq 2\sqrt{\frac{\e}{\mathrm{Lip}(f)}},
\end{equation*}
using \eqref{lambdaxi} and $|\xi| < s/2$.
A similar calculation shows that \eqref{speed1} and \eqref{gest} hold for $t \in (\xi,s)$ too. Finally, these bounds are also true for $|t| > s$ by
\eqref{h't-e}, since then $g'(t)=h'(t)$.

We now prove that $\xi$, $s$, $\rho$, $\nu$, $\sigma$, $L$, $\phi$,
$\psi$ satisfy the conditions of Lemma~\ref{lempreiss}.

We clearly have $|\xi| < s < \rho$, $0 < \nu < \frac{1}{32}$,
$\sigma> 0$ and $L > 0$. From \eqref{h't-e} and \eqref{speed1} we have $\mathrm{Lip}(h) \leq 2$ and $\mathrm{Lip}(g) \leq 2$. Hence, by \eqref{defphipsi}, $\mathrm{Lip}(\phi) + \mathrm{Lip}(\psi) \leq 4 \mathrm{Lip}(f) = L$. Further, if $|t| \geq s$ then $g(t) = h(t)$ so that $\phi(t) = \psi(t)$.

Now as $\xi \in (-s/2,s/2)$,
\begin{align}
|\phi(\xi)-\psi(\xi)| &= |f(x+\lambda')-f(x+\xi e)|\notag\\
\label{eqphiks}
&\geq |f(x+\lambda) - f(x+\xi e)| - \mathrm{Lip}(f) \|\lambda-\lambda'\|\notag\\
&\geq 240\e s - \frac{\e s}{16} \geq 160 \e s
\end{align}
by \eqref{est160}.
Hence $\phi(\xi) \neq \psi(\xi)$.

From \eqref{defphipsi} and the definition of $h$, we see that the derivative $\psi'(0)$ exists and equals $f'(x,e)$. For $|t| \leq \rho = s\sqrt{\frac{L}{2\e}}$, we have from \eqref{eq80lip}
\begin{equation*}
|f(x+te)-f(x)-f'(x,e)t| \leq \frac{\e^{2}}{160\mathrm{Lip}(f)}|t|,
\end{equation*}
so that, together with \eqref{h't-e},
\begin{align*}
|\psi(t)-\psi(0)&-t\psi'(0)|
= |f(h(t))-f(x)-f'(x,e)t|\\
&\leq |f(x+te)-f(x)-f'(x,e)t| + \mathrm{Lip}(f)\|h(t)-x-te\| \\
&\leq \frac{\e^{2}}{160\mathrm{Lip}(f)}|t| + \frac{\e^{2}}{160\mathrm{Lip}(f)}|t|
=
\sigma L |t|.
\end{align*}

Finally, using \eqref{eqphiks},
\begin{equation*}
s\sqrt{\frac{sL}{\nu|\phi(\xi)-\psi(\xi)|}} \leq
s\sqrt{\frac{sL}{\frac{1}{80}(160\e s)}} = \rho,
\end{equation*}
\begin{equation*}
\nu^{3}\left(\frac{|\phi(\xi)-\psi(\xi)|}{sL} \right)^{2} \geq
\frac{1}{80^{3}} \left(\frac{160\e s}{sL}\right)^{2} = \sigma.
\end{equation*}

Therefore, by Lemma~\ref{lempreiss},
there exists $\tau \in (-s,s)\setminus\{\xi\}$ such that
$\phi'(\tau)$ exists and
\begin{equation}
\phi'(\tau)
\geq \psi'(0) + \nu|\phi(\xi)-\psi(\xi)|/s
\geq f'(x,e) + 2\e > 0
\label{eq2k}
\end{equation}
using \eqref{eqphiks} and $\psi'(0) = f'(x,e) \geq 0$. Further, by \eqref{phipsilempreiss}
\begin{equation}\label{phiftau}
|(\phi(\tau+t)-\phi(\tau))-(\psi(t)-\psi(0))|
\leq
5\sqrt{(\phi'(\tau)-f'(x,e))L}|t|
\end{equation}
for every $t \in \real$.

From \eqref{gest} and $\e < \mathrm{Lip}(f)/9$ we have $g'(t) \neq 0$
for any $t\in (-s,s)\setminus\{\xi\}$. Define
\begin{equation}
x' = g(\tau)
\textrm{ and }
e' = g'(\tau)/\|g'(\tau)\|.
\end{equation}
The point $x'$ belongs to
$$g((-s,s)\setminus \{\xi\}) = (x-s_1,x+\lambda') \cup (x+\lambda',x+s_2).$$
Further, since the function
$\phi$ is differentiable at $\tau$, the
directional derivative $f'(x',e')$
exists and equals $\phi'(\tau)/\|g'(\tau)\|$. Now by \eqref{speed1}, \eqref{eq2k} and $\text{Lip}(\phi) \leq 2\text{Lip}(f)$ we have
\begin{equation*}
\|g'(\tau)\| \le
\frac{2\phi'(\tau)}{\phi'(\tau)+f'(x,e)},
\end{equation*}
so that
\begin{equation}\label{loestf'x'}
f'(x',e')-f'(x,e) \geq \frac{\phi'(\tau)-f'(x,e)}{2}.
\end{equation}
Hence \eqref{derbound} follows from \eqref{eq2k}.

Together with $L =4\mathrm{Lip}(f)$ and the definitions of $\phi,\psi,x'$, the inequalities \eqref{phiftau} and \eqref{loestf'x'} give
\begin{align}\label{estphif}
&|(f(g(\tau+t))-f(x')-(f(h(t))-f(x))|\\
&\notag \leq 20\sqrt{(f'(x',e')-f'(x,e))\mathrm{Lip}(f)}|t|.
\end{align}

Using \eqref{h't-e}, \eqref{gest} and $\e \leq \mathrm{Lip}(f)$ we obtain
\begin{align*}
&\|g(\tau+t)-g(\tau)-te\| \leq 3\sqrt{\frac{\e}{\mathrm{Lip}(f)}}|t|,\\
&\|h(t)-h(0)-te\| \leq \sqrt{\frac{\e}{\mathrm{Lip}(f)}}|t|
\end{align*}
for all $t$.
Using $g(\tau) = x'$, $h(0) = x$ and the Lipschitz property of $f$,
\begin{align*}
&|f(g(\tau+t))-f(x'+te)| \leq 3\sqrt{\e\mathrm{Lip}(f)}|t|,\\
&|f(h(t))-f(x+te)| \leq \sqrt{{\e}{\mathrm{Lip}(f)}}|t|
\end{align*}
for all $t$.

Putting these together with \eqref{estphif} we get
\begin{align*}
&|(f(x'+te)-f(x')-(f(x+te)-f(x))|\\
&\leq 20\sqrt{(f'(x',e')-f'(x,e))\mathrm{Lip}(f)}|t| + 3\sqrt{{\e}{\mathrm{Lip}(f)}}|t| + \sqrt{{\e}{\mathrm{Lip}(f)}}|t|\\
&\leq 25\sqrt{(f'(x',e')-f'(x,e))\mathrm{Lip}(f)}|t|
\end{align*}
as $\e \leq f'(x',e')-f'(x,e)$. This is \eqref{estff'}. We are done.
\end{proof}

\begin{lemma}[Differentiability Lemma]\label{lemdiff}
Let $H$ be a real Hilbert space, $f\colon  H \rightarrow \real$ be a Lipschitz function and
$(x,e)\in H \times S(H)$ be such that the directional
derivative $f'(x,e)$ exists and is non-negative.
Suppose that there is a family of
sets $\{F_{\e}\subseteq H \mid\e>0\}$ such that
\begin{enumerate}
\item \label{cond-one}
whenever $\e,\eta > 0$ there exists
$\delta_*=\delta_*(\e,\eta) > 0$ such that for any $\delta \in(0, \delta_*)$ and $u_1,u_2,u_3$ in the closed unit ball of $H$, one can
find $u'_1,u'_2,u'_3$ with $\|u'_m-u_m\| \leq \eta$
and
\begin{equation*}
[x+\delta u'_1,x+\delta u'_3] \cup [x+\delta u'_3,x+\delta u'_2] \subseteq F_{\e},
\end{equation*}
\item\label{cond-two}
whenever $(x',e') \in F_\e \times S(H)$ is such that
the directional derivative $f'(x',e')$ exists,
$f'(x',e') \geq f'(x,e)$ and
\begin{align}\label{xx'}
&|(f(x'+te)-f(x'))-(f(x+te)-f(x))|\\
\notag &\leq
 25\sqrt{(f'(x',e')-f'(x,e))\mathrm{Lip}(f)}|t|
\end{align}
for every $t \in \real$ then
\begin{equation}\label{ineqless}
f'(x',e') < f'(x,e) + \e.
\end{equation}
\end{enumerate}
Then $f$ is \frechet{} differentiable at $x$ and its derivative
$f'(x)$ is given
by the formula
\begin{equation}\label{fgat}
f'(x)(h) = f'(x,e) \langle h,e \rangle
\end{equation}
for $h \in H$.
\end{lemma}
\begin{proof}
We may assume $\mathrm{Lip}(f) = 1$.
Let $\e \in (0,1/9)$. It is enough to show there exists $\Delta>0$ such that
\begin{equation}\label{difR}
|f(x+ru)-f(x)-f'(x,e)\langle u,e \rangle r| < 1000\e^{1/2} r
\end{equation}
for any $u \in S(H)$ and $r\in(0,\Delta)$.

We know that the directional derivative $f'(x,e)$ exists so that there exists $\Delta > 0$ such that
\begin{equation}\label{diffxe}
|f(x+te)-f(x)-f'(x,e)t| < \frac{\e^2}{160}|t|
\end{equation}
whenever $|t| < 8\Delta/\e$. We may pick $\Delta < \delta_*(\e,\e^2/320) \e^{1/2}/4$.

Assume now, for a contradiction, that there exist $r \in (0,\Delta)$ and $u \in S(H)$ such that the inequality
\eqref{difR} does not hold:
\begin{equation}\label{nondif}
|f(x+ru)-f(x)-f'(x,e)\langle u,e\rangle r|
\geq 1000\e^{1/2} r.
\end{equation}

Define $u_1 = -e$, $u_2 = e$, $u_3 = \e^{1/2} u/4$, $s = 4\e^{-1/2}r$, $\xi=\langle u,e\rangle r$ and $\lambda=ru$. From $\|u_m\| \leq 1$, condition
\eqref{cond-one} of the present Lemma and
\begin{equation*}
s < 4\e^{-1/2}\Delta < \delta_*(\e,\e^2/320),
\end{equation*}
there exist $u_1',u_2',u_3'$ with $\|u_m'-u_m\| \leq \e^2/320$ and
\begin{equation}\label{inFset}
[x-s_1,x+\lambda'] \cup [x+\lambda',x+s_2] \subseteq F_{\e},
\end{equation}
where $s_1 = -su_1'$, $s_2 = su_2'$ and $\lambda' = s u_3'$.

We check that the assumptions of Lemma~\ref{lemmax}
hold for $f$, $\e$, $x$, $e$, $s$, $\xi$, $\lambda$, $s_1$, $s_2$,
$\lambda'$ in the Banach space $X=H$. First we note \eqref{eq80lip} is immediate from \eqref{diffxe} as $s \sqrt{2/\e} < 8r/\e < 8\Delta/\e$. We also
have $|\xi| \leq r < s/2$ as $\e < 1$. Further $|\xi| \leq r < 8\Delta/\e$ so that we may apply \eqref{diffxe} with $t = \xi$. Combining this inequality with \eqref{nondif} we obtain
\begin{equation*}
|f(x+ru)-f(x+\xi e)|
\geq 1000\e^{1/2} r - \frac{\e^2}{160}|\xi|
> 960 \e^{1/2} r
= 240\e s.
\end{equation*}
Hence \eqref{est160}. As
$\|\lambda-\xi e\|
= r\|u-\langle u,e\rangle e\|
\leq r \leq s\sqrt{\e}$ we deduce \eqref{lambdaxi}.

Now observe that for $\pi=\pm1$,
\begin{equation*}
\frac{\pi se+\lambda}{\pi s+\xi}
= e + \frac{r}{\pi s + \xi}(u-\langle u,e \rangle e)
\end{equation*}
and, as the vectors $e$ and $u-\langle u,e \rangle e$ are orthogonal and $\|\pi s + \xi\| \geq s/2$, we obtain
$$
\left\|\frac{\pi se+\lambda}{\pi s+\xi}\right\|
\leq 1 + \frac{1}{2} \frac{r^2}{(s/2)^2}
= 1 + \frac{\e}{8}.
$$
This proves \eqref{estup}.

Since $\|u_m'-u_m\| \leq \e^2/320$, \eqref{ests1s2} follows from the definitions of $u_1,u_2,s_1,s_2$. Further as $\lambda' = su_3'$ and $\lambda = ru = su_3$ we have $\|\lambda'-\lambda\| \leq s \e^2 /320 \leq \e s/16$. Hence \eqref{estlam}.

Therefore by Lemma~\ref{lemmax} there exists $x' \in [x-s_1,x+\lambda'] \cup [x+\lambda',x+s_2]$ and $e' \in S(H)$ such that $f'(x',e')$ exists, is at least $f'(x,e)+\e$ and such that \eqref{estff'} holds. But $x' \in F_{\e}$ by \eqref{inFset}. This contradicts condition \eqref{cond-two} of the present Lemma. Hence the result.
\end{proof}

\section{Proof of main result}
\label{sec5}
Let $n \geq 2$ and $\mset_{i} \subseteq \real^{n}$ ($i \in\setseq$) be
given by \eqref{defmset}.

Recall that, by Theorem~\ref{resultofsection}~(i)--(ii), the sets $\mset_i$
are closed, have Lebesgue measure zero and $\mset_i \subseteq \mset_j$ if $i \preceq j$. Here $(\setseq,\preceq)$ is a non-empty, chain complete poset that is dense and has no minimal elements, by Lemma~\ref{posetproperties}.

The following theorem shows that if $g\colon  \real^{n} \to \real$ is Lipschitz the points of differentiability of $g$ are dense in the set

$$
\mset
=
\bigcup_{\overset{i\in\setseq}{i\prec(1,1,1,\dots)}} \mset_i.
$$

\begin{theorem}\label{thmmain}
If $k,l \in\setseq$ with $k \prec l$ and
$y \in \mset_{k}$, $d > 0$ then
for any Lipschitz function $g\colon  \real^{n} \to \real$
there exists a point $x$ of \frechet{} differentiability of
$g$ with $x \in \mset_{l}$ and $\|x-y\| \leq d$.
\end{theorem}

\begin{proof}
We may assume $\mathrm{Lip}(g) > 0$. Let $H$ be the Hilbert space $\real^n$. As in Section~\ref{sec3}, for a Lipschitz function $h\colon  \real^n \to \real$ and $i \in \setseq$ we let $D^{h}_i$ be the set of pairs $(x,e) \in \mset_i \times S^{n-1}$ such that the directional derivative $h'(x,e)$ exists.

Take $i_{0} \in (k,l)$ and $j_{0} = l$.
By Theorem~\ref{resultofsection}~(iii) we can find a line segment
$\ell \subseteq \mset_{i_{0}} \cap B(y,d/2)$ of positive length.
The directional derivative of $g$ in the direction of $\ell$ exists for almost every point on $\ell$, by Lebesgue's theorem,
so that we can pick a pair $(x_{0},e_{0}) \in D^{g}_{i_{0}}$ with
$\|x_{0}-y\| \leq d/2$. Set $f_{0} = g$, $K = 25\sqrt{2\mathrm{Lip}(g)}$,
$\delta_{0} = d/2$ and $\mu = \mathrm{Lip}(g)$.

Let the Lipschitz function $f$,
the pair $(x,e)$,
the element of the index set $i \in (i_0,l)$ and, for each $\e > 0$,
the positive number $\delta_{\e}$ and the index $j_{\e} \in (i,l)$
be given by the conclusion of Theorem~\ref{thincr}. We verify
the conditions of the Differentiability Lemma~\ref{lemdiff}
hold for the function $f\colon  \real^n \to \real$, the pair $(x,e) \in D^f_i$ and the family of sets $\{F_{\e} \subseteq \real^n \mid \e > 0\}$ where
\begin{equation*}
F_{\e} =\mset_{j_{\e}} \cap B(x,\delta_{\e}).
\end{equation*}

We know from Theorem~\ref{thincr} that the derivative $f'(x,e)$ exists and is non-negative.
To verify condition \eqref{cond-one} of Lemma~\ref{lemdiff}, we may take $\e > 0$, $\eta \in (0,1)$ and put
$$
\delta_* =
\min(\alpha(i,j_{\e},\eta),\delta_{\e}/2),
$$
where $\alpha(i,j_\e,\eta)$ is given by Theorem~\ref{resultofsection}~(iii), noting $\delta(1+\eta) < 2\delta_* \leq \delta_{\e}$
for every $\delta\in(0,\delta_*)$.
Condition \eqref{cond-two} of Lemma~\ref{lemdiff}
is immediate from the definition of
$F_{\e}$ and equation \eqref{eqincr} as $\mathrm{Lip}(f) \leq \mathrm{Lip}(g) + \mu = 2\mathrm{Lip}(g)$ so that
$25\sqrt{\mathrm{Lip}(f)} \leq K$.

Therefore, by Lemma~\ref{lemdiff} the function
$f$ is differentiable at $x$.
So too, therefore, is $g$ as $g-f$ is linear.
Finally, note that  $x \in \mset_i \subseteq \mset_l$ and
\begin{equation*}
\|x-y\| \leq \|x-x_{0}\| + \|x_{0}-y\| \leq \delta_0 + d/2 = d.
\end{equation*}
\end{proof}

\begin{corollary}\label{corfinal}
If $n \geq 2$ there exists a compact subset $S \subseteq \real^{n}$ of measure $0$ that contains a point of \frechet{} differentiability of every Lipschitz function $g\colon  \real^{n} \rightarrow \real$.
\end{corollary}

\begin{proof}
Let $l \in \setseq$. As $l$ is not minimal we can find $k \prec l$.
Now $\mset_{k} \neq \emptyset$ so that we may pick
$y \in \mset_{k}$. Let $S = \mset_{l} \cap \overline{B(y,d)}$ where $d > 0$.
We know $S$ is closed and has measure zero. As it is bounded it is
also compact. If $g\colon  \real^{n} \rightarrow \real$ is Lipschitz then by
Theorem~\ref{thmmain} we can find a point $x$ of differentiability of
$g$ with $x \in \mset_{l}$ and $\|x-y\| \leq d$, so that $x \in S$.
\end{proof}

\bib
\end{document}